\documentclass[reqno,twoside,11pt]{amsart}
\usepackage{amsmath, amsfonts, amssymb, esint, amsthm, nicefrac, bbm,
  dsfont, comment}
\usepackage{epsfig, graphicx,xcolor}
\newcommand{\dx}{~\mathrm{d}x}
\newcommand{\sym}{\mathrm{sym}}

\newcommand{\Id}{\mathrm{Id}}

\newcommand{\R}{\mathbb{R}}
\def\endproof{\hspace*{\fill}\mbox{\ \rule{.1in}{.1in}}\medskip }
\newcommand*{\dbar}[1]{\bar{\bar{#1}}}
\newcommand*{\tbar}[1]{\bar{\dbar{#1}}}

\numberwithin{equation}{section}
\theoremstyle{plain}

\newtheorem*{theorem*}{Theorem}
\newtheorem{theorem}{Theorem}[section]
\newtheorem{lemma}[theorem]{Lemma}
\newtheorem{corollary}[theorem]{Corollary}

\theoremstyle{definition}

\newtheorem{example}[theorem]{Example}
 
\begin{document}
\title [Isometric immersion system in arbitrary dimension and codimension]
{Flexibility of the isometric immersion system in arbitrary dimension
  and codimension and the energy scaling of prestrained
thin films}
\author{Marta Lewicka and Hui Li}
\address{M.L.: University of Pittsburgh, Department of Mathematics, 
139 University Place, Pittsburgh, PA 15260, USA}
\address{H.L.: Changzhou Vocational Institute of Industry Technology,
School of Information Engineering,
28 Minxin Middle Rd., Changzhou, Jiangsu 213164, China}
\email{lewicka@pitt.edu, lihui@ciit.edu.cn} 

\date{\today}
\thanks{
M.L. was partially supported by NSF grant DMS-2006439.
}

\begin{abstract} 
We prove that, for a given $\mathcal{C}^{r,\beta}$-regular Riemann metric posed on a
$d$-dimensional domain, every short immersion into the
Euclidean space $\mathbb{R}^{d+k}$, can be uniformly approximated by exact
isometric immersions of regularity $\mathcal{C}^{1,\alpha}$ for any
$\alpha<\alpha_0=\min\{\frac{r+\beta}{2}, \frac{1}{1+d(d+1)/k}\}$.
% where $d_*=\frac{d(d+1)}{2}$ is the Janet dimension.   
Our theorem recovers several previously known results as special
cases. The novelty thereof lies in providing a unified flexibility statement for
arbitrary dimensions $d$ and codimensions $k$, while 
also treating the so far uncharted range $k\in (1,
\frac{d(d+1)}{2}-d+1)\setminus \{d\}$, where no corresponding general
result was previously available. Our threshold flexibility
exponent $\alpha_0$ agrees with that
obtained in \cite{lew_conv} for the closely related Monge-Amp\`ere system.
As an application, we prove a new estimate in the quantitative
immersability of thin prestrained films, setting the scaling
exponent of the infimum of non-Euclidean energies in presence of an
arbitrary prestrain metric, and in the limit of the film's vanishing thickness,
at $\frac{4\alpha_0}{\alpha_0+1}$. 
\end{abstract} 

\maketitle
%\tableofcontents

\section{Introduction}\label{sec_intro}

In this paper, we prove a new result in the theory of {\em convex
  integration} for the system of {\em isometric immersions} and, as an
application, derive a new estimate in
the {\em quantitative immersability} of thin {\em non-Euclidean films}. 
The central departing
question is whether, for a given Riemannian metric $g$ posed on a
domain $\omega\subset\R^d$, every short immersion into $\mathbb{R}^{d+k}$, where 
$k\geq 1$ denotes the codimension, can be approximated by exact
{isometric immersions} $u$:
\begin{equation}\label{II}
\begin{split}  
& (\nabla u)^T\nabla u =g \quad \mbox{in }\; \omega\subset\R^d,
\\ & \mbox{for } \; u:\omega\to\R^{d+k}, \quad g:\omega\to\R^{d\times d}_{\sym, >}.
\end{split}
\end{equation}
The seminal works by Nash \cite{Nash1, Nash2} and Kuiper
\cite{Kuiper} established that, already in codimension $k=1$, any short immersion in the
sense that $(\nabla u)^T\nabla u <g$ on the chart, is the uniform limit
of solutions to (\ref{II}) of regularity $\mathcal{C}^1$. The analogous
{\em flexibility result} in the H\"older classes $\mathcal{C}^{1,\alpha}$, for every
$\alpha<\frac{1}{1+d(d+1)}$, was later provided by Conti, De Lellis and Szekelyhidi
\cite{CDS}. Since then, considerable effort has been devoted to
determining the maximal H\"older exponent $\alpha_0$ for which
flexibility persists and it is natural to expect that
$\alpha_0$ should increase with $k$. A number of partial
results support this intuition. In codimension $k=1$,
the best result currently available is due to Inauen
\cite{In2026} and it is set at $\alpha_0=\frac{1}{2d-1}$. At the
other extreme, Lewicka proved in \cite{lew_full2d2k} that for $d=2,
k=2$, flexibility holds for every $\alpha<1$, i.e. up to
$\alpha_0=1$. This full flexibility phenomenon was subsequently
extended in the context of the closely related
Monge-Amp\`ere system in \cite{CHIL}, for the codimension
$k=\frac{d^2-d +2}{2}$. We also mention the
result by K\"allen \cite{Kallen}, showing {\em full flexibility} for (\ref{II}) whenever $k$ is
sufficiently large. For the intermediate
case $k=d$, Cao and Szekelyhidi \cite{CS} established flexibility up to
$\alpha_0=\frac{1}{2+d}$. Despite these advances, no general
result has so far provided a consistent unified description of $\alpha_0$ for all $d$ and $k$.

\begin{table}[ht]
\centering
\label{taula1}
\begin{tabular}{||c|c||c||}
\hline
{codimension $k$} & {exponent $\alpha_0$} &  {reference} \\ \hline
$1$        & $0$           & \cite{Nash1, Nash2, Kuiper} \\ \hline  
$2d_*$   & $1$           & \cite{Kallen}      \\ \hline  
$1$        & ${1}/({1+2d_*})$       & \cite{CDS}      \\ \hline
$1$        & ${1}/(1+d^2-d)$      & \cite{CHI2}  \\ \hline
$1$        & ${1}/({2d-1})$          & \cite{In2026}  \\ \hline  
$d$        & ${1}/({2+d})$           & \cite{CS}      \\ \hline  
$d_*-d+1$      & $1$                & \cite{lew_full2d2k, CHIL}      \\ \hline  
$k\geq 1$       & ${1}/({1+ 2d_*/k})$           & this paper      \\ \hline
\end{tabular}
\vspace{2mm}
\caption{Development of convex integration results for system of
  isometric immersions (\ref{II}) in the general dimension $d\geq 2$
  and the metric $g\in\mathcal{C}^2$.}
\vspace{-4mm}
\end{table}

\noindent This is what we achieve in the present paper. We prove
flexibility of the system (\ref{II}), namely the density of its
$\mathcal{C}^{1,\alpha}$ solutions in the set of
subsolutions (the {\em short immersions}) for all $\alpha<\alpha_0=\frac{1}{1+2d_*/k}$.
Our theorem recovers several previously known results as special
cases. When $k=1$, it yields the exponent obtained in \cite{CDS}, although
it does not reach the improved threshold in \cite{In2026}. When
$k\to\infty$, it is consistent with the full flexibility 
in \cite{Kallen}, while remaining weaker than the sharp findings of
\cite{lew_full2d2k, CHIL}. When
$d=k$, it includes the result of \cite{CS}. The novelty of
our work lies also in treating the so far uncharted range $k\in (1,
\frac{d(d+1)}{2}-d+1)\setminus\{d\}$, where no corresponding general
theorem was previously available. Moreover, our exponent agrees with that
obtained in \cite{lew_conv} for the Monge-Amp\`ere system.
Our main result is the following:

\begin{theorem}\label{th_final}
Let $g\in\mathcal{C}^{r,\beta}(\bar\omega, \R^{d\times d}_{\sym,>})$ be
defined on the closure of an open set $\omega\subset\R^d$,
diffeomorphic to $B_1$, for some $r+\beta>0$. Then, for every
immersion $u\in\mathcal{C}^1(\bar\omega, \R^{d+k})$ satisfying:
$$(\nabla u)^T\nabla u< g \quad \mbox{in }\; \bar\omega,$$
for every $\epsilon>0$ and for every regularity exponent $\alpha$ in:
\begin{equation}\label{range}
0<\alpha < \min\Big\{\frac{r+\beta}{2}, \frac{1}{1+2d_*/k}\Big\} 
\quad \mbox{where} \quad d_*=\frac{d(d+1)}{2},
\end{equation}
there exist $\bar u\in\mathcal{C}^{1,\alpha}(\bar\omega,\R^{d+k})$, such that:
\begin{equation}\label{outcome}
\|\bar u - u\|_0\leq \epsilon \quad \mbox{and} \quad (\nabla
\bar u)^T\nabla \bar u =g \quad \mbox{in }\;\bar\omega. 
\end{equation}
\end{theorem}

\medskip

\noindent We point out that the first condition in (\ref{outcome}) can
be replaced by $\|\bar u - u\|_{0,\bar\alpha}\leq \epsilon$ for any
fixed independent $\bar\alpha\in (0,1)$, so that, in fact, a subsolution $u$ to
(\ref{II}) is the limit in $\mathcal{C}^{0,\bar\alpha}(\bar\omega,
\R^{d+k})$ of the exact isometric immersions $\bar u$ of regularity
$\mathcal{C}^{1,\alpha}(\bar\omega, \R^{d+k})$ with the exponent $\alpha$ as
in (\ref{range}).
It is currently an important open problem how to combine the various
convex integration techniques that have been developed by the
authors in recent years, see Table \ref{taula1}. At present, these methods appear to be
tailored to the specific dimension and
codimension scenarios, and it remains unclear how to interpolate
between the best results available at $k=1$ and $k=\frac{d(d+1)}{2}-d+1=d_*-d+1$,
or whether further improvements can be achieved by a suitable
synthesis of these approaches. We refer the reader to
\cite{lew_full2d2k} for a broader overview of the field and a historical
account of the quest for the optimal flexibility exponent $\alpha_0$. 

\bigskip

\noindent As an application, we additionally prove a new estimate in
the {\em quantitative immersability}
of thin {\em non-Euclidean films}. More precisely, given $\omega\subset\R^d$, define the
family of ``thin films'':
\begin{equation}\label{omh}
\Omega^h = \big\{(x,z); ~x\in\omega, ~ z\in B_h\subset\R^k\big\}\subset\R^{d+k},
\end{equation}
parametrized by their ``thickness'' $h\ll 1$.
Consider the Riemannian metric $g$ on $\Omega^1$ and 
pose the problem of minimizing the following energy functionals, as  $h\to 0$:
\begin{equation}\label{Eh}
\mathcal{E}_g^h(u) = \frac{1}{h^k}\int_{\Omega^h} W\big((\nabla
u)g^{-1/2}\big)\det g^{1/2}~\mathrm{d}(x,z)
\qquad\mbox{ for all }\; u\in H^{1}(\Omega^h,\R^{d+k}).
\end{equation}
In general, the energy density $W:\R^{(d+k)\times (d+k)}\to [0,\infty]$ is assumed
to be $\mathcal{C}^2$-regular in the vicinity of $\mathrm{SO}(d+k)$, equal to
$0$ at $\Id_{d+k}$, and frame-invariant in the sense that $W(RF)=
W(F)$ for all $R\in \mathrm{SO}(d+k)$.
When $W\sim \mbox{dist}^2(\cdot, \mathrm{SO}(d+k))$, the energy in (\ref{Eh})
measures the averaged pointwise deficit of a deformation $u$ from
being an orientation preserving isometric immersion of
$g$ on $\Omega^h$. For $d=2$, $k=1$, it models the
{\em elastic energy} of deformations of a prestrained film $\Omega^h$, and various
techniques have been applied to its study \cite{lew_book}, including
questions on asymptotics of the minimizing configurations as
$h\to 0$, in function of the scaling exponent $\theta$ in: $\inf \mathcal{E}_g^h\sim
h^\theta$. Our second main result is the following:

\begin{theorem}\label{th_elasticity}
Let $g\in\mathcal{C}^{r,\beta}(\bar\Omega^1, \R^{(d+k)\times (d+k)}_{\sym,>})$ be
defined on the thin films $\Omega^h$ in (\ref{omh}) where
$\omega\subset\R^d$ is diffeomorphic to $B_1$, and where
$r+\beta>0$. Then, there holds:
$$\inf \mathcal{E}^h_g\leq Ch^{\frac{4\alpha}{\alpha+1}}
\qquad\mbox{for all } \alpha \mbox{ in the range (\ref{range})}, $$
with $C$ depending only on $\omega, g, d, k, \alpha$ but not on $h$.
\end{theorem}

\medskip

\noindent We point out that for $d=2, k=1$, the asymptotic behavior as $h\to 0$ of the
minimizing sequences to (\ref{Eh}) under the assumption that the
prestrain metric $g$ is smooth, has been fully classified in all viable
scaling regimes corresponding to $h^\theta$ with $\theta\geq 2$ (see \cite{lew_book}).
Since the current state of the art sets the flexibility exponent at
$\alpha_0=\frac{1}{3}$ in virtue of \cite{CHI2}, resulting in $\frac{4\alpha_0}{\alpha_0+1}=1$,
Theorem \ref{th_elasticity} consequently implies the following
uncharted scaling regime in which questions on the asymptotics and the arising types of
singularities should be studied:
$$\inf \mathcal{E}_g^h\sim h^\theta \quad\mbox{ where }\quad 1\leq\theta<2.$$
On the other hand, for $d=2, k=2$, flexibility holds all the way up to
$\alpha_0=1$ in virtue of \cite{lew_full2d2k}, so that 
$\frac{4\alpha_0}{\alpha_0+1}=2$, which yields
$\inf \mathcal{E}_g^h \leq C h^\theta$ for all
$\theta<2$. Consequently, there is no ``gap'' between the
applicability of the membrane energy and the Kirchhoff bending energy,
similarly to the situation of thin rods (enjoying the same codimension $k=2$).

\subsection{The main technical contribution} \label{sub_over}

Our result relies on the {\em stage} construction,
whose iteration via the Nash-Kuiper algorithm in Theorem \ref{thm_NK}
yields Theorem \ref{th_final}.  
To ease the notation, we define the {\em defect}
relative to a matrix field $g$ and a vector field $u$:
$$\mathcal{D}(g, u) \doteq g - (\nabla u)^T\nabla u.$$
 
\begin{theorem}\label{thm_STA} {\textup{[STAGE]}}
Let $g\in \mathcal{C}^{r,\beta}(\bar\omega, \R^{d\times d}_{\sym,>})$  be
defined on the closure of an open set $\omega\subset\R^d$
diffeomorphic to $B_1$, for some regularity exponents:
\begin{equation}\label{regu}
0<r+\beta\leq 2.
\end{equation}
Fix an immersion $\underline u\in\mathcal{C}^\infty(\bar\omega,\R^{d+k})$
and a constant $\underline\gamma\geq 1$ such that:
\begin{equation}\label{imma}
\frac{1}{\underline\gamma}\Id_d\leq (\nabla \underline u)^T\nabla \underline u\leq
\underline\gamma\Id_d\quad\mbox{ in } \;\bar\omega.
\end{equation}
Then, there exists
$\underline \delta\in (0,1)$ and $\underline\sigma>1$, depending only
on $\omega, g, \underline u, \underline\gamma, d,k$, such that the following holds.
Given any $u\in\mathcal{C}^2(\bar\omega,\R^{d+k})$ and any $\delta,
\lambda, \sigma$ such that: 
\begin{align*}
& \delta\leq \underline\delta,\qquad \lambda\delta^{1/2}\geq 1,
\qquad\sigma\geq \underline\sigma,\qquad \sigma^{J+3}\delta^{1/2}\leq 1,
\tag*{(\theequation)$_1$}\refstepcounter{equation} \label{Ass1} \\
& \|\nabla u - \nabla \underline u\|_0 \leq \frac{\rho_{\underline\gamma}}{2},
\tag*{(\theequation)$_2$} \label{Ass2}\\
& \|\mathcal{D}(g-\delta H_0, u)\|_0\leq\frac{r_0}{4}\delta \quad\mbox{
and }\quad \|u\|_2\leq \lambda\delta^{1/2},
\tag*{(\theequation)$_3$} \label{Ass3}
\end{align*}
where $r_0, H_0$ depending only on $d$, and given in Lemma \ref{lem_dec_def}, 
and $\rho_{\underline\gamma}$ depending only on $\underline\gamma, d, k$ is
given in Lemma \ref{lem_propa}, 
there exists $\tilde u\in\mathcal{C}^2(\bar\omega,\R^{d+k})$ satisfying:
\begin{align*}
& \|\tilde u-u\|_1\leq C\delta^{1/2},\qquad \|\tilde u\|_2\leq C\lambda\sigma^J\delta^{1/2},
\tag*{(\theequation)$_1$}\refstepcounter{equation}\label{Res1}\\
&\Big\|\mathcal{D}\Big(g-\frac{\delta}{\sigma^{S}}H_0,\tilde u\Big)\Big\|_0\leq
  \frac{r_0}{5}\frac{\delta}{\sigma^{S}} + \frac{\|g\|_{r,\beta}}{\lambda^{r+\beta}},
\tag*{(\theequation)$_2$} \label{Res2}
\end{align*}
with constants $C$ depending only on $\omega, g, \underline u, \underline\gamma, 
d,k$ and where the integers $J, S$ are given through:
\begin{equation}\label{def_JS}
Jk=Sd_*=  lcm (d_*, k).
\end{equation}
\end{theorem}

\medskip

\noindent The above yields the decay rate $S$ of 
$\mathcal{D}$ and the blow-up rate $J$  of $\nabla^2 u$, having the quotient:
\begin{equation*}
\frac{J}{S}= \frac{lcm(d_*,k)}{k} \cdot \frac{d_*}{lcm(d_*,k)} = \frac{d_*}{k}.
\end{equation*}
The H\"older regularity of the limiting immersion deduced from iterating the
{\em stage}, depends only on the aforementioned quotient and the regularity
of $g$, through the formula in (\ref{ran_NK}). We now quote two
auxiliary statements: (i) the {\em pre-stage} which reduces
the given defect to the defect sufficiently small to allow the
applicability of (ii) the {\em Nash-Kuiper iteration scheme}. Both
statements were proved in \cite[Theorem 4.1, Theorem 1.3]{lew_full2d2k} for the
specific case $d=k=2$, they are however independent of
the particular values of $d,k\geq 2$. We note that the effective
assumption on the codimension $k\geq 2$ is only needed to 
employ an alternative {\em step} formula, namely Nash's spirals (rather than
Kuiper's corrugations as in Lemma \ref{lem_step2}), in order
to carry out the {\em initial stage}, yielding an immersion that
satisfies a specific scaling law involving
its first and second derivatives, and the norm of the defect. 
Since the codimension $k=1$ case has been treated in \cite{In2026}, Theorem
\ref{th_final} follows from the two statements below:

\begin{theorem}\label{lem_prestage}\textup{[PRE-STAGE]}
Let $g\in \mathcal{C}(\bar\omega, \R^{d\times d}_{\sym, >})$
be a Riemann metric posed on the closure of an open set $\omega\subset\R^d$ diffeomorphic to $B_1$,
and let $u\in\mathcal{C}^1(\bar\omega, \R^{d+k})$ be an immersion, with:
$$\mathcal{D}(g, u)>0 \quad\mbox{ in } \; \bar\omega.$$ 
Then, there exist an immersability parameter $\underline \gamma \geq 1$
depending only on $g, u$, such that for
every $\epsilon, \delta>0$ there exists $\underbar
u\in\mathcal{C}^\infty(\bar\omega,\R^{d+k})$, satisfying:
\begin{equation*}
\begin{split}
& \frac{1}{\underline \gamma} \Id_d\leq (\nabla \underbar u)^T\nabla
\underbar u \leq \underline\gamma\Id_d \quad\mbox{ and } \quad
\mathcal{D}(g, \underbar u)>0 \quad\mbox{ in } \,\bar\omega,\\
& \|u-\underbar u\|_0\leq \epsilon \quad\mbox{ and }\quad
\|\mathcal{D}(g, \underbar u)\|_0\leq \delta.
\end{split}
\end{equation*}
\end{theorem}

\begin{theorem}\label{thm_NK} {\textup{[NASH-KUIPER'S ITERATION]}}
Let $g\in \mathcal{C}^{r,\beta}(\bar\omega, \R^{d\times d}_{\sym,>})$
and $\underbar u\in\mathcal{C}^\infty(\bar\omega, \R^{d+k})$ be
defined on the closure of an open set $\omega\subset\R^d$ diffeomorphic to $B_1$,
for some regularity exponents in (\ref{regu}) and the immersability
constant $\underline\gamma\ge \max\{1, \|g\|_0\}$ as in (\ref{imma}).
Assume that:
$$\mathcal{D}(g,\underline u) >0 \quad\mbox{in }\,\bar\omega
\quad \mbox{ and } \quad \|\mathcal{D}(g,\underline u)\|_0\leq
\bar\rho_{\underline\gamma},$$
where $\bar\rho_{\underline\gamma}$ is a constant, depending on $\underline\gamma, d$.
Assume further that for some $S,J,p>0$ and for:
$$\underline \delta\in (0,1), \quad \underline\sigma>1$$ 
depending only on $\omega, g, \underline u, \underline \gamma, d, k, S, J, p$, the following holds:
\begin{equation}\label{STA}
\left[\qquad  \begin{minipage}{11cm}
Given any $u\in\mathcal{C}^2(\bar\omega,\R^{d+k})$ and any $\delta, \lambda, \sigma$ such that:
\begin{equation*}
\begin{split}
& \delta\leq \underline\delta,\qquad \lambda\delta^{1/2}\geq 1,
\qquad\sigma\geq \underline\sigma,\qquad \sigma^p\delta^{1/2}\leq 1,\\ 
& \|\nabla u - \nabla \underline u\|_0\leq \frac{\rho_{\underline\gamma}}{2},\\
& \|\mathcal{D}(g-\delta H_0, u)\|_0\leq\frac{r_0}{4}\delta \quad\mbox{
and }\quad \|u\|_2\leq \lambda \delta^{1/2},
\end{split}
\end{equation*}
with $r_0, H_0$ as in Lemma \ref{lem_dec_def} and
$\rho_{\underline\gamma}$ as in Lemma \ref{lem_propa},
there exists $\tilde u\in\mathcal{C}^2(\bar\omega,\R^{d+k})$ satisfying:
\begin{equation*}
\begin{split}
& \|\tilde u-u\|_1\leq C\delta^{1/2},\qquad \|\tilde u\|_2\leq C\lambda \sigma^J\delta^{1/2},\\
&\Big\|\mathcal{D}\Big(g-\frac{\delta}{\sigma^S}H_0,\tilde u\Big)\Big\|_0\leq
  \frac{r_0}{5}\frac{\delta}{\sigma^S} + \frac{\|g\|_{r,\beta}}{\lambda^{r+\beta}},
\end{split}
\end{equation*}
with constants $C$ depending only on $\omega, g, \underline u,
\underline\gamma, d, k, S, J, p$.
\end{minipage}\quad \right]
\end{equation}
Then, for every $\epsilon>0$ and every exponent $\alpha$ in the range:
\begin{equation}\label{ran_NK}
0<\alpha<\min\Big\{\frac{r+\beta}{2},\frac{1}{1+ 2J/S}\Big\},
\end{equation}
there exists an immersion $\bar u\in\mathcal{C}^{1,\alpha}(\bar\omega, \R^{d+k})$ such that:
\begin{equation}\label{outcome2}
\|\bar u - \underline u\|_0\leq \epsilon\quad\mbox{ and }\quad
\mathcal{D}(g,\bar u)=0\quad\mbox{in }\; \bar\omega.
\end{equation}
\end{theorem}

\medskip

\noindent The first condition in (\ref{outcome2}) may
be replaced by $\|\bar u - \underline u\|_{0,\bar\alpha}\leq \epsilon$, for any
fixed independent $\bar\alpha\in (0,1)$. Indeed, $\bar u$ is obtained as the limit of a
sequence of immersions $\{u_n\in\mathcal{C}^2(\bar\omega,\R^{d+k})\}_{n\to\infty}$,
converging in $\mathcal{C}^{1,\alpha}$ and obtained by successive
applications of the stage statement (\ref{STA}). For a given
$\epsilon>0$, the first immersion $u_0$ is chosen to satisfy:
$$\|u_0-\underline u\|_0\leq\frac{\epsilon}{2},\qquad
\|\nabla u_0- \nabla \underline u\|_0\leq C,$$
with constant $C$ depending only on $\omega, \underline u, g$, whose
magnitude is essentially the order $1$ quantity with respect to $\|\mathcal{D}(g,\underline
u)\|_0^{1/2}$. The subsequent immersions satisfy then the inductive bounds:
$$\|u_{n+1}-u_n\|_1\leq C\delta_n^{1/2}, \quad \mathcal{D}(g,u_n)\leq C\delta_n
\quad \mbox{ with }\quad C\sum_{n=0}^\infty\delta_n^{1/2}
\leq \frac{\epsilon}{2}.$$
Hence, $\{u_n\}_{n\to\infty}$ is a Cauchy sequence in
$\mathcal{C}^1$ (it is such also in $\mathcal{C}^{1,\alpha}$) and
the limiting $\bar u$ satisfies:
$$\|\bar u-\underline u\|_0\leq\epsilon,\qquad
\|\nabla \bar u - \nabla \underline u\|_0\leq C,$$
which implies, by the interpolation inequality with respect to $\bar\alpha<1$:
$$\|\bar u-\underline u\|_{0,\bar\alpha}\leq C \|\bar u-\underline u\|_{1}^{\bar\alpha}
\|\bar u-\underline u\|_{0}^{1-\bar\alpha}\leq C \epsilon^{1-\bar\alpha}.$$
This implies that a subsolution $\bar u$ is the limit in $\mathcal{C}^{0,\bar\alpha}(\bar\omega,
\R^{d+k})$ of the exact isometric immersions $\bar u$ of regularity
$\mathcal{C}^{1,\alpha}(\bar\omega, \R^{d+k})$, with the exponent $\alpha$ limited
according to the range (\ref{range}).

\bigskip

\noindent We also remark that the same result as in Theorem \ref{th_final}, remains valid in any target
space $\R^n$ replacing $\R^{d+k}$, for $n\geq d+k$. Our
proofs only require existence of and the propagation bounds on $k$
normal vector fields, which are guaranteed by
having $n-d\geq k$. The same statement as in Theorem
\ref{thm_STA} also holds and one concludes the final result by a
version of Theorem \ref{thm_NK}.

\medskip

\subsection{Comparison with the literature} Our present result
contributes to the growing body of work on the
flexibility of $\mathcal{C}^{1,\alpha}$ isometric immersions and
solutions to related PDEs. 
The study of flexibility in the context of the system (\ref{II}) originates from the
classical Nash-Kuiper theorem \cite{Nash1, Kuiper} about
$\mathcal{C}^1$ isometric immersions of a $d$-dimensional Riemannian manifold
into $\mathbb{R}^{d+1}$ (in our terminology, this corresponds to
codimension $k=1$). The H\"older-regular isometric
immersions were studied by Borisov \cite{Borisov1958, Borisov1965}, who conjectured that
the Nash-Kuiper theorem extended to $\mathcal{C}^{1,\alpha}$-regularity for
$\alpha < \frac{1}{1 + d^2 + d}$ and $\alpha < \frac{1}{5}$ for $d = 2$.  
He provided a proof for the exponent $\alpha <\frac{1}{7}$, dimension $d=2$
and the analytic metric $g$ in \cite{Borisov2004}. 
Borisov's conjecture was confirmed by Conti, De Lellis and Sz\'ekelyhidi
in \cite{CDS}, whereas the case
$\alpha < \frac{1}{5}$, $d = 2$ was resolved by De Lellis, Inauen and Sz\'ekelyhidi in \cite{DIS1/5}.  
These results were generalized to compact manifolds by Cao and
Sz\'ekelyhidi in \cite{CaoSze2022}. Locally, it has been recently improved to 
$\alpha < \frac{1}{d^2 - d + 1}$ by Cao, Hirsch and Inauen in \cite{CHI2}
and to $\alpha<\frac{1}{2d-1}$ by Inauen in \cite{In2026}.

\medskip

\noindent We also mention the explicit convex integration constructions and numerical visualizations
of $\mathcal{C}^1$ isometric embeddings, developed by
Borrelli and collaborators, in particular in the following contexts: for flat tori \cite{B1}, reduced spheres
\cite{B2} and the hyperbolic plane \cite{B3}. These works give a
concrete geometric picture of the underlying fine-scale corrugations. 

\medskip

\noindent If the codimension $k$ is larger than $1$, one can construct more regular isometric
immersions. In this direction, Nash first showed \cite{Nash2} that any
manifold $d$-dimensional Riemannian manifold with metric $g \in
\mathcal{C}^r$, $r \geq 3$, admits a
$\mathcal{C}^r$-regular isometric immersion into $\mathbb{R}^{d+k}$
for sufficiently large $k$.  
The codimension bounds were later established and refined by Gromov
\cite{GromovPdr} and G\"unther \cite{Guenther}, while the case of  $g \in
\mathcal{C}^{r,\beta}$ with $r+\beta>2$, was was treated by Jacobowitz \cite{Jaco}.
When the metric's regularity is lower, i.e. $
r + \beta<2$, K\"all\'en \cite{Kallen} constructed a
$\mathcal{C}^{1,\alpha}$ immersion for any $\alpha < \frac{r + \beta}{2}$
when $k \geq 2d_*(d+1)$ is sufficiently large. 
See also \cite{DI2020,CaoIn2024,CaoSz2025} for further
results about isometric immersions in high codimension. 
Recently, Lewicka proved in  \cite{lew_full2d2k} that in dimension
$d=2$ and codimension $k\geq 2$, every
subsolution to (\ref{II}) with the metric $g\in\mathcal{C}^{r,\beta}$
admits an approximating sequence of $\mathcal{C}^{1,\alpha}$ isometric
immersions, for all $\alpha<\min\{\frac{r+\beta}{2}, 1\}$.

\medskip

\noindent When $d=2, k=1$, we quote the following rigidity
results. Firstly, according to the classical 
statements due to Cohn-Vossen \cite{CV} and Herglotz
\cite{Her}, in the resolution of the Weyl problem \cite{Weyl}, a
$\mathcal{C}^2$ isometric immersion of $\mathbb{S}^2$ into $\R^3$ must
be a rigid motion, i.e. a composition
of a rotation and a translation. Second, any $\mathcal{C}^2$ surface
with positive Gauss's curvature must be locally convex; for
generalizations of this statement we refer to \cite{HarNir}, and 
to \cite[Chapter II]{Pogo1}, \cite[Chapter IX]{Pogo2}
where the requirement of the $\mathcal{C}^2$ regularity is
replaced by the requirement that the immersion is $\mathcal{C}^1$ 
and that the measure on $\mathbb{S}^2$ induced by the Gauss map has
bounded variation.
%and that the image of the Gauss map has measure zero in $\mathbb{S}^2$. 
Third, using geometric arguments, the same rigidity has been
proved by Borisov in a series of papers \cite{Borisov1958, Borisov1965}, for
$\mathcal{C}^{1,\alpha}$ isometric immersions at $\alpha>\frac{2}{3}$, of
$\mathcal{C}^2$ metrics, with a simpler
analytic proof of this result obtained in \cite{CDS}.
In particular, we see that $\frac{2}{3}$ is an upper bound
on the range of H\"older exponents that can be reached using convex
integration for isometric immersions of $2$-dimensional metrics into
$\R^{2+1}$. We point out that no such
rigidity statement is possible for the isometric immersions already
into $\R^4$, in view of \cite{lew_full2d2k}.

\medskip

\noindent The parallel version of the Nash-Kuiper scheme as in the papers
cited above, was first developed to handle 
flexibility of $\mathcal{C}^{1,\alpha}$ weak solutions to the 
Monge-Amp\`ere equation by Lewicka and Pakzad in
\cite{lewpak_MA}, where the result was proved for
$\alpha<\frac{1}{7}$. This regularity exponent was later increased to
$\alpha<\frac{1}{5}$ and $\alpha<\frac{1}{3}$ by Cao and Szekelyhidi \cite{CS}, and
by Cao, Hirsch and Inauen \cite{CHI}, respectively. Further,
in \cite{lew_conv} Lewicka introduced the Monge-Amp\`ere system
and obtained flexibility of its solutions for any 
$\alpha<\min\{\frac{\beta}{2}, \frac{1}{1+2d_*/k}\}$ given the right hand side field
$A\in\mathcal{C}^{0,\beta}$, where the dimension $d$ and $k$ were arbitrary;
and also for any $\alpha<\min\{\frac{\beta}{2}, 1\}$ provided that $k\geq 2d_*$.
For the special case $d=2$, weak solutions of progressively higher
regularity, depending on the value of $k$, were obtained in the sequence of
papers \cite{lew_improved, lew_improved2, in_lew}. 
In \cite{in_lew2}, Inauen and Lewicka established the full flexibility
of the Monge-Amp\'ere system for $d=k=2$, and in \cite{CHIL} their
result was extended to arbitrary $d$ and the corresponding $k=d_*-d+1$.

\smallskip

\subsection{Organization of the paper} 

Section \ref{sec_step} begins by recalling the standard mollification
estimates and the lemma on the (local) decomposition of the defect into rank-one
primitive defects. We then introduce the new
{\em step} construction, based on the Kuiper corrugation ansatz from
\cite{lew_conv},  lifted here to the fully nonlinear setting of
(\ref{II}). Its proof is deferred to section \ref{sec_appA}. We
also present several auxiliary lemmas concerning the propagation of the normal frame to the
immersions generated in the inductive course of the {\em stage}. As a byproduct
of our estimates, we deduce existence of the normal frame to any at least
$\mathcal{C}^1$-regular immersion, with its derivatives obeying a
natural bound with multiplicative constants
that depend only on the immersability parameter and the modulus of
continuity of the tangent space. These results are of independent interest and are proved in section \ref{sec_appB}. 

\medskip

\noindent In section \ref{sec_step} we also provide an example showing
that, in codimension $k\geq 2$, the aforementioned constant in the
normal frame estimates cannot, in general, be chosen independently of
the immersion. Specifically, we construct a sequence of immersions
$u_n$ with uniform immersability parameter, and satisfying $\|\nabla u_n\|_1\leq Cn$,
whereas every normal frame $E_n$ to $u_n$ has $\|\nabla
E_n\|_0\geq Cn^2$, with $C$ independent of $n$ and
of the choice of frame. 

\medskip

\noindent Section \ref{sec_khamsa} is devoted
to the proof of Theorem \ref{thm_STA}, which through the
Nash-Kuiper iteration and the pre-stage constructions, stated in Theorems
\ref{thm_NK} and \ref{lem_prestage}, respectively, yields our main Theorem \ref{th_final}.
The modified immersion $\tilde u$ produced in a single stage is obtained through
$N=lcm(d_*, k)$ steps (i.e. the least common multiple of the Janet dimension
and the codimension). Each step applies
the Kuiper corrugation to the effect of replacing one rank-one primitive component of the defect $\mathcal{D}$ by higher order error terms. Since $\mathcal{D}$
consists of $d_*$ such components, every time the step counter
goes over a multiple of $d_*$, the defect decreases
by one factor of $\sigma$. On the other hand, whenever the step counter goes over a multiple of
$k$, the  size of $\nabla^2u$ increases, also essentially by one
factor of $\sigma$. Consequently, after $N$ steps one obtains the
fundamental blow-up/decay
ratio $J/S = d_*/k$. Throughout the construction, each step necessitates
the update of the normal frame to the modified immersion, and
components of  $\nabla^2u$ are then estimated separately along the
various normal directions, exploiting the bounds on the
propagation of the normal frame given in the auxiliary lemmas.

\medskip

\noindent Finally, in section \ref{sec_nonEu}
we provide the proof of Theorem \ref{th_elasticity} by  applying
a suitable version of the Kirchhoff-Love extension, formulated with
respect to the ambient
metric $g$, to the H\"older regular isometric immersions constructed
in the preceding sections.

\smallskip

\subsection{Notation} 

\noindent By $\mathbb{R}^{d\times d}_{\sym}$ we denote the space of symmetric
$d\times d$ matrices, and $\mathbb{R}^{d\times d}_{\sym, >}$ stands for the set of such matrices which are positive definite. The space of H\"older continuous vector fields
$\mathcal{C}^{m,\alpha}(\bar\omega,\R^k)$ consists of restrictions of
all $f\in \mathcal{C}^{m,\alpha}(\mathbb{R}^d,\R^k)$ to the closure of
an open, bounded set $\omega\subset\R^d$. The
$\mathcal{C}^m(\bar\omega,\R^k)$ norm of this restriction is
denoted by $\|f\|_m$, while its H\"older norm in $\mathcal{C}^{m,
  \alpha}(\bar\omega,\R^k)$ is $\|f\|_{m,\alpha}$. By $\nabla^{(m)}f$
we denote the table-valued field of all partial derivatives of $f$ of order
$m$. If $d=k$, the symmetric part of $\nabla f$ is: $\sym\nabla f = (\nabla f + (\nabla f)^T)/2$.
By $C$ we denote a universal constant that may change from line to line.
% but it depends only on the specified parameters.

\section{The preparatory statements}\label{sec_step}

In this section, we collect several technical ingredients that will be
used in the proof of Theorem \ref{thm_STA}. 
The first lemma concerns the convolution and commutator estimates from \cite{CDS}.

\begin{lemma}\label{lem_stima}
Let $\phi\in\mathcal{C}_c^\infty(\R^d,\mathbb{R})$ be a standard
mollifier that is nonnegative, radially symmetric, supported on the
unit ball $B(0,1)\subset\R^d$ and such that $\int_{\mathbb{R}^d} \phi \dx = 1$. Denote: 
$$\phi_l (x) \doteq \frac{1}{l^d}\phi(\frac{x}{l})\quad\mbox{ for all
}\; l\in (0,1], \;  x\in\R^d.$$
Then, for every $f,g\in\mathcal{C}^0(\mathbb{R}^d,\R)$, every
$m\geq 0$ and $r\in\{0,1\}$, $\beta\in (0,1]$, there holds:
\begin{align*}
& \|\nabla^{(m)}(f\ast\phi_l)\|_{0} \leq
\frac{C}{l^m}\|f\|_0,\tag*{(\theequation)$_1$}\vspace{1mm} \refstepcounter{equation} \label{stima1}\\
& \|f - f\ast\phi_l\|_0\leq l^{r
+\beta} \|f\|_{r,\beta},\tag*{(\theequation)$_2$} \vspace{1mm} \label{stima2}\\
& \|\nabla^{(m)}\big((fg)\ast\phi_l - (f\ast\phi_l)
(g\ast\phi_l)\big)\|_0\leq {C}{l^{2- m}}\|\nabla f\|_{0}
\|\nabla g\|_{0}, \tag*{(\theequation)$_3$} \label{stima4}
\end{align*}
with constants $C>0$ depending only on $d$ and $m$. 
\end{lemma}

\medskip

\noindent The next lemma provides a decomposition of symmetric
matrices into linear combinations of rank-one {\em primitive matrices}.

\begin{lemma}\label{lem_dec_def}
There exist the unit vectors $\{\eta_i\in\R^d\}_{i=1}^{d_*}$, the
linear maps $\big\{\bar a_{i}:\R^{d\times d}_\sym\to
\R\big\}_{i=1}^{d_*}$, and $r_0\in(0,1)$, depending only on $d$, such that
$\R^{d\times d}_\sym = \mbox{span}\{\eta_i\otimes \eta_i\}_{i=1}^{d_*}$, with:
$$H=\sum_{i=1}^{d_*}\bar a_{i}(H) \eta_{i}\otimes \eta_{i}
\quad\mbox{ for all } H\in \R^{d\times d}_\sym,$$
and such that, denoting: $H_0 \doteq \sum_{i=1}^{d_*}\eta_{i}\otimes
\eta_{i}$, we have:
$$|\bar a_{i} (H)-1|\leq\frac{1}{2} \quad \mbox{ for all } H\in
B(H_0,r_0)\subset \R^{d\times d}_{\sym}, \; i=1\ldots d_*,$$
\end{lemma}

\noindent The above result is elementary. Define 
$\{\eta_i\}_{i=1}^{d_*}$ as $ \big\{\xi_{ij}\doteq
\frac{e_i+e_j}{|e_i+e_j|}\big\}_{i,j=1\ldots d,\;i\leq j}$. Then,
the set $\{\xi_{ij}\otimes\xi_{ij}\}_{i,j=1\ldots d,\;i\leq j}$
provides a basis of $\R^{d\times d}_\sym$. Since the set's cardinality
matches the dimension of  $\R^{d\times d}_\sym$, its elements must be
linearly independent, and $\bar a_i(H_0)=1$ for all $i=1\ldots d_*$.

\medskip

\noindent We now present the {\em step} in our convex integration algorithm, in
which a single codimension is used to cancel one rank-one defect of
the form $a(x)^2\eta\otimes \eta$.
The construction utilizes the ansatz that is the
basic building block of the convex integration step for the Monge-Amp\`ere system
in \cite[Lemma 2.1]{lew_conv}. The observation that a similar ansatz also applies in the present fully
nonlinear setting of (\ref{II}), was first made in \cite[Lemma
3.1]{CHI2} and led to a simplification of the analysis as
well as an improvement of the regularity
exponent pertaining to codimension $k=1$, and 
revealed the further close connection between the Monge-Amp\`ere
and the isometric immersion systems.
The proof of the lemma below will be given in section \ref{sec_appA}.
We also note that the version of the ansatz, given in \cite[Lemma
3.1]{CHI2} and \cite[Lemma 2.5]{lew_full2d2k}, featuring an
unbalanced form of the leading terms in \eqref{sisi}, is not sufficient
for our purposes. 

\begin{lemma}\label{lem_step2} \textup{[STEP: KUIPER'S CORRUGATION]}
Let $u\in\mathcal{C}^2(\R^d,\R^{d+k})$ be an immersion and let $E_u
\in\mathcal{C}^{1}(\R^d, \R^{d+k})$ be a unit normal vector field to $u$, so that:
$$(\nabla u)^T E_u = 0,\quad |E_u|=1 \quad\mbox{ in } \R^d.$$
For a given unit vector $\eta\in\R^d$ we set $t=\langle x, \eta\rangle$, and denote:
$$\Gamma(t) = \sqrt{2}\sin t, \quad \bar\Gamma(t) = -\frac{1}{4}\sin(2t), \quad 
\dbar\Gamma(t) = \frac{1}{4}\cos(2t), \quad 
\tbar\Gamma(t) = 1-\frac{1}{2}\cos(2t). $$
Then, for every $\lambda>0$ and $a\in\mathcal{C}^2(\R^d, \R)$, the
vector field $\tilde u\in \mathcal{C}^1(\R^d, \R^{d+k})$ in:
\begin{equation*}
\begin{split}
& \tilde u(x) = u(x)+ \frac{\Gamma(\lambda t )}{\lambda} a(x) E_u(x) + 
\frac{\bar\Gamma(\lambda t )}{\lambda}a(x)^2 T_u(x)\eta 
+ \frac{\dbar\Gamma(\lambda t )}{\lambda^2}a(x) T_u(x)\nabla a(x),
\\ & \mbox{where: } T_u = (\nabla u)\big((\nabla u)^T\nabla
u\big)^{-1}\in\mathcal{C}^1(\R^d, \R^{{(d+k)}\times d}),
\end{split}
\end{equation*}
satisfies the following identity:
\begin{equation}\label{sisi}
\begin{split}
& (\nabla \tilde u)^T\nabla \tilde u - (\nabla u)^T\nabla u -
a^2\eta\otimes \eta \\ & = -\frac{2\Gamma(\lambda
  t)}{\lambda} a\;\big[\langle \partial^2_{ij}u,E_u\rangle\big]_{i,j=1\ldots d} 
+\frac{2\dbar\Gamma(\lambda t)}{\lambda^2}a\nabla^2a 
+ \frac{\tbar\Gamma(\lambda t)}{\lambda^2}\nabla a\otimes \nabla a+ \mathcal{R}.
\end{split}
\end{equation}
The error term $\mathcal{R} = \mathcal{R}(\lambda, a, \eta,\nabla u,
E_u) $ above is written as:
\begin{equation*}
\mathcal{R} =  \mathcal{R}_0 +  \mathcal{R}_1 + \mathcal{R}_2 + \mathcal{R}_3 + \mathcal{R}_4,
\end{equation*}
whose respective components, grouped by the powers of $\lambda^{-1}$, are given by:
\begin{equation*}
\begin{split}
 & \mathcal{R}_0 = \bar\Gamma'(\lambda t)^2a^4|T_u\eta|^2\eta\otimes\eta,
\\ &  \mathcal{R}_1 = - \frac{2\bar\Gamma(\lambda t)}{\lambda} a^2\; \big[
\langle \partial^2_{ij}u,T_u\eta\rangle\big]_{i,j=1\ldots d} 
+ \frac{2\bar\Gamma(\lambda t)\Gamma'(\lambda t)}{\lambda}
a\;\sym\big(\nabla (a^2 T_u \eta)^T (E_u\otimes \eta)\big)  
\\ & \qquad +\frac{2\Gamma(\lambda t)\bar\Gamma'(\lambda t)}{\lambda}
a^2\;\sym\big(\nabla(aE_u)^T T_u(\eta\otimes\eta) \big)
+ \frac{2\bar\Gamma(\lambda t)\bar\Gamma'(\lambda t)}{\lambda}
a^2\;\sym\big(\nabla( a^2T_u\eta)^TT_u(\eta\otimes\eta)\big) 
\\ & \qquad + \frac{2\bar\Gamma'(\lambda t)\dbar\Gamma'(\lambda t)}{\lambda}
a^3\;\big\langle T_u\eta, T_u\nabla a\big\rangle \eta\otimes\eta,
\\ &\mathcal{R}_2 = - \frac{2\dbar\Gamma(\lambda t)}{\lambda^2}a\;\big[ \langle
\partial^2_{ij}u,T_u\nabla a\rangle\big]_{i,j=1\ldots d}
+\frac{2\dbar\Gamma(\lambda t)\Gamma'(\lambda
  t)}{\lambda^2}a\;\sym\big(\nabla(aT_u\nabla a)^T(E_u\otimes \eta) \big) 
\\ & \qquad  + \frac{2\bar\Gamma'(\lambda t)\dbar\Gamma(\lambda
  t)}{\lambda^2}a^2\;\sym\big(\nabla(aT_u\nabla a)^T T_u(\eta\otimes\eta) \big)
+\frac{\Gamma(\lambda t)^2}{\lambda^2} a^2(\nabla E_u)^T\nabla E_u
\\ & \qquad  + \frac{2\Gamma(\lambda t)\bar\Gamma(\lambda
  t)}{\lambda^2}\sym\big(\nabla(aE_u)^T \nabla(a^2T_u\eta) \big) 
+ \frac{2\Gamma(\lambda t)\dbar\Gamma'(\lambda t)}{\lambda^2}a\;\sym\big(\nabla(aE_u)^T
T_u(\nabla a\otimes \eta) \big)
\\ & \qquad + \frac{\bar\Gamma(\lambda t)^2}{\lambda^2}\nabla(a^2T_u\eta)^T \nabla(a^2T_u\eta)
+\frac{2\bar\Gamma(\lambda t)\dbar\Gamma'(\lambda
  t)}{\lambda^2}a\;\sym\big(\nabla(a^2T_u\eta)^T T_u(\nabla a\otimes \eta) \big) 
\\ & \qquad + \frac{\dbar\Gamma'(\lambda t)^2}{\lambda^2}a^2\; |T_u\nabla a|^2 \eta\otimes\eta,\\
& \mathcal{R}_3 = \frac{2\bar\Gamma(\lambda t)\dbar\Gamma(\lambda
  t)}{\lambda^3}\sym\big(\nabla(a^2T_u\eta)^T \nabla(aT_u\nabla a) \big)
+ \frac{2\Gamma(\lambda t)\dbar\Gamma(\lambda t)}{\lambda^3}\sym\big(\nabla(a E_u)^T
\nabla(aT_u\nabla a) \big)
\\ & \qquad +\frac{2\dbar\Gamma(\lambda t)\dbar\Gamma'(\lambda
  t)}{\lambda^3}a\;\sym\big(\nabla(aT_u\nabla a)^TT_u(\nabla a\otimes \eta) \big),
\\ &  \mathcal{R}_4 =
\frac{\dbar\Gamma(\lambda t)^2}{\lambda^4} \nabla(aT_u\nabla a)^T \nabla(aT_u\nabla a). 
\end{split}
\end{equation*}
\end{lemma}

\smallskip

\noindent Finally, we collect several lemmas concerning propagation (and existence) of
the normal frame to the immersions constructed inductively in our convex
integration algorithm. The proofs, which follow the arguments of \cite[Section 3]{lew_full2d2k}
in the case $d=k=2$, are presented in section \ref{sec_appB}.

\begin{lemma}\label{lem_det}
Let $u\in\mathcal{C}^1(\bar\omega,\R^{d+k})$ defined on the closure of an
open, bounded $\omega\subset\R^d$, satisfy:
\begin{equation}\label{immers_gamma}
\frac{1}{\gamma}\Id_d\leq (\nabla u)^T\nabla u \leq \gamma\Id_d
\quad\mbox{ in } \bar\omega,
\end{equation}
for some $\gamma\geq 1$. Then there holds:
$$(d/\gamma)^{1/2}\leq |\nabla u|\leq (d\gamma)^{1/2} \quad\mbox{ and }\quad 
\frac{1}{\gamma^d}\leq\det((\nabla u)^T\nabla u)\leq \gamma^d\quad 
\mbox{ in } \, \bar\omega.$$
\end{lemma}

\begin{lemma}\label{lem_Tu}
Given an immersion $u\in\mathcal{C}^{n+1}(\bar\omega,\R^{n+k})$ posed on the closure of an
open, bounded set $\omega\subset\R^d$ with  (\ref{immers_gamma}) valid
for some immersability constant $\gamma\geq 1$, consider the tangent
field $T_u=(\nabla u)\big((\nabla u)^T\nabla 
u\big)^{-1}\in\mathcal{C}^{n}(\bar\omega,\R^{(d+k)\times d})$. Then there holds:
\begin{itemize}
\item[(i)] $\|T_u\|_0\leq C$ and $\displaystyle{\|\nabla^{(m)}T_u\|_0\leq C
\hspace{-4mm} \sum_{p_1+2p_2+\ldots mp_m=m} \; \prod_{i=1}^m\|\nabla^{(i)}\nabla u\|_0^{p_i}}$
for all $ m=1\ldots n$, where $C$ depends only on $\gamma, d, k, n$.
\item[(ii)] If $n\geq 1$ and with some constants $\mu>1$, $A<1$ and
  $\bar C>1$, we additionally have:
$$\|\nabla^{(m)}\nabla u\|_0\leq \bar C \mu^{m}A \quad \mbox{ for all } \; m=1\ldots n,$$
then there holds:
$$\|\nabla^{(m)}T_u\|_0\leq C\mu^m A \quad\mbox{ for all } m=1\ldots n,$$
where $C$ depends only on $\gamma, \bar C, d, k, n$, but not on $\mu, A$.
\end{itemize}  
\end{lemma}

\medskip

\noindent The following result provides the key construction and estimates
on the propagation of normal frame from a given, to a nearby
immersion. The same construction, but with less precise bounds 
appeared in \cite{DI2020}[Proposition 5.3], while the same result at
$d=k=2$ appeared in \cite{lew_full2d2k}. 

\begin{lemma}\label{lem_propa}
Let $u\in\mathcal{C}^{n+1}(\bar\omega, \R^{d+k})$ be defined on the closure of an
open, bounded set $\omega\subset\R^d$, and satisfy (\ref{immers_gamma}) for
some $\gamma\ge 1$. Let $\{E_u^i\in\mathcal{C}^{n}(\bar\omega, \R^{d+k})\}_{i=1}^k$
be a normal frame of $u$, namely:
\begin{equation}\label{norm_frame}
(\nabla u)^T E_u^i = 0,\quad |E^i_u|=1,
\quad \langle E^i_u, E^j_u\rangle = 0 \quad\mbox{ in } \;\bar\omega,
\quad \mbox{ for all }\, i,j=1\ldots k, \; i\neq j.
\end{equation}
There exists a constant $\rho_\gamma\in (0,1)$ depending only on
$\gamma, d$, such that for every $v\in\mathcal{C}^{n+1}(\bar\omega, \R^{d+k})$, satisfying
$\|\nabla v-\nabla u\|_0\leq \rho_\gamma$, there holds:
\begin{itemize}
\item[(i)] $v$ is an immersion and its normal frame
$\{E_v^i\in\mathcal{C}^{n}(\bar\omega, \R^{d+k})\}_{i=1}^k$, namely:
$$(\nabla v)^T E_v^i = 0,\quad |E^i_v|=1,
\quad \langle E^i_v, E^j_v\rangle = 0 \quad\mbox{ in } \;\bar\omega,
\quad \mbox{ for all }\, i,j=1\ldots k, \; i\neq j,$$ 
can be defined via the following formulas:
\begin{equation}\label{propa_def} 
\begin{split}
& E^i_v = \frac{\nu^i_v - \sum_{j<i}\langle \nu_v^i, E^j_v\rangle
  E^j_v}{|\nu^i_v - \sum_{j<i}\langle \nu_v^i, E^j_v\rangle E^j_v|}
\qquad \mbox{ for all }\; 1=1\ldots k,\\ 
& \mbox{where } \nu_v^i= \big(\Id_{d+k} - T_v(\nabla v - \nabla u)^T\big)E^i_u
\quad\mbox{for all } i=1\ldots k \\ & \mbox{and} \quad T_v= (\nabla v)((\nabla v)^T\nabla v)^{-1}.
\end{split}
\end{equation}
\item[(ii)] For all $m=0\ldots n$ and $i=1\ldots k$ there holds, with
  constants $C$ depending on $\gamma, d, k, n$:
\begin{equation*}
\begin{split}    
& \|\nabla^{(m)}(E^i_v-E^i_u)\|_0\\ & \leq C \hspace{-7mm}\sum_{p+\sum_{s=1}^ms(p_s+\bar
  p_s + \sum_{j\leq i}\dbar p_{j,s})=m\hspace{-21mm}} \hspace{-7mm} \|\nabla^{(p)}(\nabla v-\nabla u)\|_0
\prod_{s=1}^m \|\nabla^{(s)}(\nabla v - \nabla u)\|_0^{p_s} \|\nabla^{(s)}\nabla v\|_0^{\bar p_s}
\prod_{j\leq i}\|\nabla^{(s)}E^j_u\|_0^{\dbar p_{j,s}}.
\end{split}
\end{equation*}
\item[(iii)]  If $n\geq 1$ and with some constants $\mu>1, A<1$ and
  $\bar C>1$ we additionally have:
\begin{equation*}
\begin{split}
& \|\nabla^{(m)}(\nabla v-\nabla u)\|_0\leq \bar C \mu^mA \quad\mbox{ for } m=0\ldots n,
\\ & \|\nabla^{(m)}\nabla u\|_0\leq \bar C\mu^{m}, ~~ \;\;
\|\nabla^{(m)}E^i_u\|_0\leq \bar C\mu^{m}  \quad\mbox{for } m=1\ldots n,\; i=1 \ldots k,
\end{split}
\end{equation*}
then there holds:
$$\|\nabla^{(m)}(E^i_v-E^i_u)\|_0\leq C \mu^mA \quad\mbox{ for all }
m=0\ldots n,\; i=1\ldots k,$$
where $C$ depends only on $\gamma, \bar C, d, k, n$, but not on $\mu, A$.
\end{itemize}
\end{lemma}

\medskip

\noindent Equipped with Lemma \ref{lem_propa}, one can deduce
existence of a normal frame to any 
$\mathcal{C}^1$ immersion $u$, with gradient bounds whose constants
depend on $\gamma$ and on the
modulus of continuity of $\nabla u$. The mere existence of a normal frame was
essentially proved in \cite[Lemma 2.4.1]{Del_masterpieces}. We have:

\begin{corollary}\label{lem_normals}
Let $u\in\mathcal{C}^{n+1}(\bar\omega, \R^{d+k})$ be an immersion defined on the closure
of the set $\omega\subset\R^d$ diffeomorphic to $B_1$, such that it satisfies
(\ref{immers_gamma}) with some $\gamma\geq 1$. Then, there
exist a normal frame $\{E^i_u\in \mathcal{C}^n(\bar\omega, \R^{d+k})\}_{i=1}^k$
to $u$, as in (\ref{norm_frame}), obeying the bounds:
\begin{equation}\label{stim_nor}
\|\nabla^{(m)}E_u^i\|_0\leq C \|\nabla u\|_m\quad\mbox{
  for all } \; m=1\ldots n, \; i=1\ldots k,
\end{equation}
where $C$ depends only on $\omega, \gamma, d, k, n$, and the modulus of
continuity of $\nabla u$ on $\bar\omega$.
\end{corollary}

\medskip

\noindent We point out that the dependence of $C$ in (\ref{stim_nor}) on  the modulus of continuity of $\nabla u$ is essential and cannot be disposed of, as the next example shows:

\begin{example}\label{ex_normals}\footnote{Prepared with assistance from ChatGPT.}
We will construct a sequence of immersions $\{u_n\in\mathcal{C}^\infty(\bar B_2,\R^4)\}_{n=1}^\infty$ such that:
\begin{equation}\label{ep1}
\frac{1}{\gamma}\Id_2\leq (\nabla u_n)^T\nabla u_n\leq \gamma\,\Id_2 \quad \mbox{ in } \, \bar B_2, \qquad \|\nabla u_n\|_{1}\leq Cn,
\end{equation}
and such that for any normal frame $\{E_n^i\in\mathcal{C}^1(\bar B_2,\R^4)\}_{i=1}^2$ to $u_n$ as in (\ref{norm_frame}), there holds:
\begin{equation}\label{ep2}
\|\nabla E_n^1\|_0 + \|\nabla E_n^2\|_0\geq Cn^2,
\end{equation}
with constants $\gamma, C$ independent of $n$ and of the choice of the frame itself.

\smallskip

{\bf 1.} We define immersions $u_n$ as rescalings of the following
immersion $f\in\mathcal{C}^\infty(\R^2,\R^4)$: 
\begin{equation}\label{ep3}
\begin{split}
& u_n(x,y) = \frac{1}{2\pi n}f(2\pi nx, 2\pi n y),
\\ & \mbox{where: } \; f(x,y) = \big(\sin x, \sin y, \cos (x+2y), \cos x+\cos y\big).
\end{split}
\end{equation}
We observe that $f$ is bounded and $2\pi$ periodic in both coordinates $x$ and $y$.
Since $\nabla u_n(x,y) = \nabla f(2\pi nx, 2\pi ny)$ and $\nabla^{(2)} u_n(x,y) = 2\pi n\nabla^{(2)} f(2\pi nx, 2\pi ny)$, there follows the second assertion in (\ref{ep1}). We also record a direct computation:
\begin{equation}\label{ep4}
\begin{split}
& \partial_xf= \big(\cos x, 0, -\sin (x+2y), -\sin x\big), \qquad 
\partial_yf = \big(0,\cos y, -2\sin(x+2y), -\sin y\big), 
\\ & \partial_{xx}f= \big(-\sin x, 0, -\cos (x+2y), -\cos x\big), \\ & \partial_{xy}f= \big(0, 0, -2\cos (x+2y), 0\big), 
\qquad  \partial_{yy}f = \big(0, -\sin y, -4\cos(x+2y), -\cos y\big).
\end{split}
\end{equation}
To prove the first assertion in (\ref{ep1}), it suffices to check that, for some $\gamma\geq 1$:
$$\frac{1}{\gamma}\Id_2\leq (\nabla f)^T\nabla f \leq \gamma\,\Id_2 \quad\mbox{ in }\,\R^2.$$
The upper bound $(\nabla f)^T\nabla f\leq 9\,\Id_2$ follows since $\|\nabla f\|_0^2\leq 9$. For the lower bound, we first apply Sylvester's criterion to conclude the positive definiteness of: 
\begin{equation*}
\begin{split}
(\nabla f)^T\nabla f = \left[\begin{array}{cc}1+\sin^2(x+2y) & 2\sin^2(x+2y)+\sin x\sin y \\  2\sin^2(x+2y)+\sin x\sin y & 1+4\sin^2(x+2y) \end{array}\right],
\end{split}
\end{equation*}
at each fixed $(x,y)$. Indeed, there holds:
\begin{equation*}
\begin{split}
& 1+\sin^2(x+2y) > 0,
\\ & \det \big((\nabla f)^T\nabla f \big)= \big(1-\sin^2x\sin^2y \big) +\sin^2(x+2y)(5-4\sin x\sin y)>0,
\end{split}
\end{equation*}
where in the last expression above both summation terms are
nonnegative, but cannot equal $0$ at the same time. This is because $\sin^2x\sin^2y=1$
only when $x,y\in \pi/2 +\pi\mathbb{Z}$,  precluding $x+2y\in \pi
\mathbb{Z}$ which is a necessary condition for $\sin(x+2y)=0$. Since
$(\nabla f)^T\nabla f $ is continuous and $[0,2\pi]^2$ periodic, the
positive definiteness condition is uniform. 

\smallskip

{\bf 2.} Fix $n\geq 1$ and let $E_n\in\mathcal{C}^1(\bar B_2,\R^{4\times 2})$ be any normal frame to $u_n$. Define the vector field $h_n\in\mathcal{C}^1(\bar B_2,\R^2)$, together with its unscaled version $h\in \mathcal{C}^1(2\pi n\bar B_2,\R^{2})$ by:
\begin{equation*}
\begin{split}
& h_n=E_n^Te_3=\big(\langle E^1_n, e_3\rangle, \langle E^2_n, e_3\rangle\big),
\\ & h = E^Te_3 = \big(\langle E^1, e_3\rangle, \langle E^2, e_3\rangle\big).
\end{split}
\end{equation*}
Here, $E(2\pi n x, 2\pi n y)=E_n(x,y)$ is the corresponding normal frame to $f$, satisfying:
$$E^TE = \Id_2 \quad \mbox{ and } \quad (\nabla f)^TE = 0 \quad\mbox{ on } \; 2\pi n \bar B_2.$$
We claim that $h$ vanishes precisely on the set $\mathcal{A}\cap 2\pi n\bar B_2$, where:
$$\mathcal{A} = \big\{(x,y); ~ \cos x =\cos y=0\big\} = \big\{\frac{\pi}{2}+\pi\mathbb{Z}\big\}^2.$$
Indeed, $h^{-1}(0)$ consists of points where $e_3$ belongs to the tangent plane of $f$, namely when $e_3\in\mathrm{span}\{\partial_xf, \partial_y f\}$. A simple calculation shows that this occurs only in $\mathcal{A}$, where we also get:
\begin{equation}\label{ep5}
\begin{split}
& E^1, E^2\in\mbox{span}\{e_1, e_2\},  \\ & \mbox{and } \quad e_3 = \frac{-(\sin y)\partial_xf + (\sin x)\partial_yf}{2-\sin x\sin y} \qquad \mbox{ for all }\, (x,y)\in \mathcal{A}\cap2\pi n\bar B_2.    
\end{split}
\end{equation}
In particular, $h_n^{-1}(0)=\frac{1}{2\pi n}\mathcal{A}\cap \bar B_2$ which is disjoint from the boundary $\partial (0,1)^2$ of the open square $(0,1)^2\Subset B_2$. We now estimate the, hence well defined, Brouwer degree $\deg\big(h_n, (0,1)^2,0\big)$ of $h_n$ with respect to $0$, utilizing the winding number formula:
$$\deg\big(h_n, (0,1)^2,0\big) = \frac{1}{2\pi}\int_{\partial(0,1)^2} \frac{\langle \partial_\tau h_n, h_n^\perp\rangle}{|h_n|^2}\;\mbox{d}s,$$
where $\tau$ is the (counterclockwise oriented) unit tangent vector to $\partial(0,1)^2$.
Observe that, since the length $|h_n|$ is simply the length of the orthogonal projection of $e_3$ onto the normal plane, it is independent of the choice of $E_n$ or $E$. 
Indeed, we may write:
$$|h_n|_{\mid\partial(0,1)^2}=|h|_{\mid\partial(0,2\pi n)^2}
= \big|\mbox{proj}_{(\nabla f)^\perp}e_3\big|_{\mid \partial(0,2\pi n)^2}\geq c_0>0.$$
The last bound above follows from the fact that the length $|h|$ is periodic with period $2\pi$ in both $x$ and $y$, as $f$ is. Recalling that the boundary $\partial(0,2\pi n)^2$ is disjoint from $\mathcal{A}$, we get that $|h|$ is strictly positive along it and bounded away from $0$, by some $c_0>0$ that depends only on $f$. 
The degree formula therefore implies the lower bound on the norm $\|\nabla E_n\|_0$: 
\begin{equation}\label{ep6}
\begin{split}
|\deg(h_n, (0,1)^2,0)|& \leq \frac{\big|\partial(0,1)^2\big|}{2\pi}\frac{\|\nabla h_n\|_0}{\min_{\partial(0,1)^2} |h_n|}  \leq \frac{2}{\pi c_0}\|\nabla h_n\|_0 \\ & = C\big\|[(\partial_pE_n)^Te_3]_{p=1,2}\big\|_0\leq C\|\nabla E_n\|_0,
\end{split}
\end{equation}
with $C$ that depends only on $f$. In the next step, we will directly compute the degree $\deg\big(h_n, (0,1)^2,0\big)$ and obtain the lower bound on (\ref{ep6}).

\smallskip

{\bf 3.} Consider the set of zeroes of $h$ in the square $(0, 2\pi n)^2$. It consists of $4n^2$ points:
$$\mathcal{A}_0= \mathcal{A}\cap (0, 2\pi n)^2 = \Big\{\big(\frac{\pi}{2}+k\pi, 
\frac{\pi}{2}+s\pi\big);~ k,s=0,\dots, 2n-1\Big\}.$$
We proceed to calculate signs of the Jacobian determinant of $h$ at each such point, needed in:
\begin{equation}\label{ep7}
\deg\big(h_n, (0,1)^2,0\big) = \deg\big(h, (0,2\pi n)^2,0\big) = 
\sum_{(x,y)\in \mathcal{A}_0}\mathrm{sgn}\,\det \nabla h (x,y). 
\end{equation}
Firstly, the formulas (\ref{ep4}) yield in $\mathcal{A}_0$:
\begin{equation*}
\begin{split}
& \partial_xf = -\sin (x+2y)e_3 -(\sin x) e_4,\quad \partial_yf = -2\sin (x+2y)e_3 -(\sin y) e_4  
\\ & \partial_{xx}f = -(\sin x)e_1, \quad \partial_{xy}f = 0, \quad\partial_{yy}f = -(\sin y)e_2\qquad \mbox{ for all } \, (x,y)\in\mathcal{A}_0.
\end{split}
\end{equation*}
Consequently, differentiating the identities $\langle E^i, \partial_xf\rangle =
\langle E^i, \partial_yf\rangle =0$ in $x$ and $y$, we get:
\begin{equation*}
\begin{split}
& \langle \partial_xE^i, \partial_xf\rangle = -\langle E^i, \partial_{xx}f\rangle = (\sin x)\langle E^i, e_1\rangle, \qquad \langle \partial_xE^i, \partial_yf\rangle = 0, \\ & \langle \partial_yE^i, \partial_xf\rangle = 0, \qquad
\langle \partial_yE^i, \partial_yf\rangle = -\langle E^i, \partial_{yy}f\rangle = (\sin y)\langle E^i, e_2\rangle \qquad \mbox{ for all } \, (x,y)\in\mathcal{A}_0.
\end{split}
\end{equation*}
Recalling the formula (\ref{ep5}) for $e_3$, the above implies that:
\begin{equation*}
\begin{split}
& \partial_x\langle E^i, e_3\rangle = \langle \partial_xE^i, e_3\rangle = -\frac{\sin x\sin y}{2-\sin x\sin y}
\langle E^i, e_1\rangle, \\ & \partial_y\langle E^i, e_3\rangle = \langle \partial_yE^i, e_3\rangle = \frac{\sin x\sin y}{2-\sin x\sin y} \langle E^i, e_2\rangle \qquad \mbox{ for all } \, i=1,2, ~(x,y)\in\mathcal{A}_0.  
\end{split}
\end{equation*}
This leads to the following formula:
\begin{equation}\label{ep8}
\begin{split}
& \det\nabla h = - \Big(\frac{\sin x\sin y}{2-\sin x\sin y} \Big)^2\det R, 
\\ & \mbox{where: } R=\left[\begin{array}{c|c} \langle E^1, e_1\rangle & \langle E^2, e_1\rangle \\ \hline \langle E^1, e_2\rangle & \langle E^2, e_2\rangle\end{array}\right]\in \mathcal{O}(2) \qquad \mbox{ for all } \, (x,y)\in\mathcal{A}_0. 
\end{split}
\end{equation}
From the above, we already see that each zero point $(x,y)\in\mathcal{A}_0$ of $h$ is indeed non-degenerate and its signature equals $-\det R$. We claim that all those signatures are, in fact, the same. To this end, consider the full frame: $\big[\partial_xf, \partial_yf, E^1, E^2\big]\in\mathcal{C}^1(2\pi n \bar B_2, \R^{4\times 4})$. Its determinant is a continuous function with a fixed sign (say, +1), as it never passes through $0$. On $\mathcal{A}_0$ that frame may be written in the block form:
$$ \big[\partial_xf, \partial_yf, E^1, E^2\big] = \left[\begin{array}{c|c} 0 & R \\ \hline \begin{array}{c|c} -\sin(x+2y) & -2\sin (x+2y) \\ \hline -\sin x & -\sin y\end{array} & 0 \end{array}\right] \quad \mbox{ for all } \, (x,y)\in\mathcal{A}_0, $$
and therefore, $\det R$ has indeed constant, positive sign at all points in $\mathcal{A}_0$:
\begin{equation*}
\begin{split}
\mbox{sgn}\,\det R & = \mbox{sgn}\, \det \left[\begin{array}{c|c} -\sin(x+2y) & -2\sin (x+2y) \\ \hline -\sin x & -\sin y\end{array}\right] \\ & = \mbox{sgn}\big(2-\sin x \sin y\big) = +1 \quad \mbox{ for all } \, (x,y)\in\mathcal{A}_0.
\end{split}
\end{equation*}
In conclusion, (\ref{ep7}) and (\ref{ep8}) imply the lower bound for (\ref{ep6}):
$$\deg\big(h_n, (0,1)^2, 0\big) = - |\mathcal{A}_0| = -4n^2,$$
which yields, with a constant that depends only on $f$, the claimed estimate (\ref{ep2}):
$$\|\nabla E_n\|_0\geq C \big|\mbox{deg}\big(h_n, (0,1)^2, 0\big)\big| \geq Cn^2.$$
\endproof
\end{example}

\section{The stage construction and a proof of Theorem
  \ref{thm_STA}} \label{sec_khamsa}  

The proof consists of several steps in an inductive construction
below. By $C$ we always denote a constant that is larger than $1$ and
depends only on the referential quantities $\omega, \underline u,
\underline \gamma, \omega, d, k$, although it
may change from line to line. In general, we will use the following counters:
\begin{equation*}
\begin{split}
& i: 0\ldots N=lcm(d_*, k):\quad \mbox{consecutive induction counter} \\ &
s: 0\ldots S:\quad \mbox{ counter on multiplicities of } d_* \\ & 
j: 0\ldots J: \quad \mbox{ counter on multiplicities of } k\\ & 
\theta: 1\ldots k: \quad \mbox{ counter on codimensions } \\ &
\rho: 1\ldots d_*: \quad \mbox{counter on defect decomposition modes } \\ & 
m\geq 0: \hspace{9.7mm} \mbox{number of derivatives } \\ &
p,q: 1\ldots d: \hspace{2mm} \mbox{differentiation directions}.
\end{split}
\end{equation*}
Our induction procedure comprises one preparatory step (at counter
value $i=0$) in which the given quantities $u, g$ are mollified to
their smooth approximations $u_0, g_0$; and $N=lcm(d_*, k)$
subsequent steps (from the counter value $i=1$ to $i=N$), in which the
preceding immersion $u_{i-1}$ is adapted to a new immersion $u_{i}$.
Since each of these finitely many steps ``loses'' only finitely many derivatives between
those whose bounds are required at $i-1$ and those whose bounds are
achieved at $i$ -- as is the case, for example, in Lemmas  
\ref{lem_normals} and \ref{lem_propa} --  all constants $C$ appearing in
the proof depend only on finitely many derivatives of the initial
quantities $u_0, g_0$. To avoid cluttering the presentation and
obscuring the argument, we shall write statements such as ``for all $m\geq 0$''
or ``for all $m\geq 1$'', in place of the more precise ``for all $m=0\ldots K$'', 
or ``for all $m=1\ldots K$''. Here, $K$ would always be sufficiently
large and depend through a rather complicated formula on the current values of the counters
$i, s, j, \theta, \rho$, as well as on $d_*$ and $k$. 

\medskip

\noindent {\bf Proof of Theorem \ref{thm_STA}}

{\bf 1. (Preparing the data)}  We will use the following assumptions
on $u\in\mathcal{C}^2(\bar\omega, \R^{d+k})$:
\begin{align*}
& \|\nabla u-\nabla \underline u\|_0\leq\frac{\rho_{\underline\gamma}}{2},
\tag*{(\theequation)$_1$} \refstepcounter{equation} \label{Asu2}\\
& \|\mathcal{D}(g-\delta H_0, u)\|_0\leq\frac{r_0}{4}\delta
\hspace{1.05cm}\mbox{and }\quad \|u\|_2\leq \lambda\delta^{1/2},
\tag*{(\theequation)$_2$} \label{Asu3}
\end{align*}
where $\underline \gamma>1$ is given, $r_0, H_0$ are as in Lemma \ref{lem_dec_def}, and
$\delta,\lambda, \sigma$ are any chosen parameters satisfying a
slightly weakened version of the assumption \ref{Ass1}:
\begin{equation}\label{Asu1}
0<\delta\leq \underline\delta,\qquad \lambda\delta^{1/2}\geq 1,
\qquad\sigma\geq \underline\sigma,\qquad \sigma^{J+2}\delta^{1/2}\leq 1.
\end{equation}
The full assumption \ref{Ass1} will be used in the final Step 7 of the proof.
The constants $\underline\delta<1$ and $\underline \sigma>1$ will be
chosen, respectively, sufficiently small and sufficiently large, in
the course of the arguments below, and will depend only on
$\omega,\underline u, \underline\gamma, d, k$ and $J,S$ that are defined in (\ref{def_JS})
(thus effectively also depending only on $d,k$). We set:
\begin{equation*}
\begin{split}
& \lambda_0=\lambda,\quad\delta_0=\delta,\quad l=\frac{1}{\bar C\lambda_0}, \\
& u_0 = u*\phi_l\in\mathcal{C}^\infty(\bar\omega, \R^{d+k}),\quad
g_0= g*\phi_l\in\mathcal{C}^\infty(\bar\omega, \R^{(d+k)\times(d+k)}_{\sym,>}),
\end{split}
\end{equation*}
for $\bar C>1$ that depends only on $d, k$ and is sufficiently large
to ensure that:
\begin{equation}\label{uahid}
\|(\nabla u_0)^T\nabla u_0 - ((\nabla u)^T\nabla u)*\phi_l\|_0\leq C
l^2\|\nabla^2u\|_0^2\leq \frac{C}{\bar
  C^2\lambda_0^2}\lambda_0^2\delta_0\leq \frac{r_0}{12}\delta_0,
\end{equation}
in view of \ref{stima4} and \ref{Asu3}. Further, Lemma \ref{lem_stima}
and \ref{Asu3} yield that:
\begin{equation}\label{ithnain}
\begin{split}
& \|g_0-g\|_0\leq l^{r+\beta}\|g\|_{r,\beta} \leq \frac{\|g\|_{r,\beta}}{\lambda_0^{r+\beta}}
\\ & \|u_0 - u\|_1\leq l\|u\|_2\leq l\lambda\delta^{1/2}\leq \delta_0^{1/2}
\\ & \|\nabla^{(m)}\nabla^{(2)}u_0\|_0\leq \frac{C}{l^m}\|\nabla^{(2)} u_0\|_0
\leq \frac{C}{l^m}\lambda\delta^{1/2}\leq C\lambda_0^{m+1}\delta^{1/2}
\quad\mbox{ for all } m\geq 0.
\end{split}
\end{equation}
In particular, recalling Lemma \ref{lem_det}, the second bound above implies:
\begin{equation*}
\begin{split}
\|(\nabla u_0)^T\nabla u_0 - (\nabla u)^T\nabla u\|_0 & \leq \|\nabla
(u_0-u)\|_0 \big(2 \|\nabla u\|_0 + \|\nabla (u_0-u)\|_0\big) \\ & \leq
\delta_0^{1/2}(2(2\underline \gamma d)^{1/2} + 1),
\end{split}
\end{equation*}
so taking $\delta_0$ sufficiently small (i.e. setting
$\underline\delta$ small enough) we ensure that $u_0$ is an immersion:
\begin{equation}\label{thalata}
\|\nabla u_0-\nabla\underline u\|_0\leq \rho_{\underline\gamma}\quad
\mbox{ and } \quad \frac{1}{3\underline\gamma}\Id_d\leq (\nabla u_0)^T\nabla u_0\leq
\, 3\underline\gamma\Id_d\; \mbox{ in } \;\bar\omega,
\end{equation}
By Corollary \ref{lem_normals} we obtain the normal frame
$\{E_{\underline u}^\theta\in\mathcal{C}^\infty(\bar\omega,
\R^{d+k})\}_{\theta=1}^k$ to $\underline u$. Then, Lemma \ref{lem_propa} yields existence
of a (sufficiently regular) normal frame
$\{E_0^\theta\doteq E^\theta_{u_0}\in \mathcal{C}(\bar\omega, \R^{d+k})\}_{\theta=1}^k$, namely:
$$ (\nabla u_0)^TE_0^\theta=0, \quad |E_0^\theta|=1,\quad \langle
E_0^\theta, E_0^{\bar\theta}\rangle =0 \quad \mbox{in } \bar\omega, \quad
\mbox{for all } \; \theta,\bar\theta=1\ldots k, \; \theta\neq\bar\theta,$$
so that the following bounds hold:
\begin{equation}\label{arbaa}
\begin{split}
& \|\nabla^{(m)}E_0^\theta\|_0\leq C (\|\nabla^{(m)}E_{\underline
  u}\|_0 +\|\nabla u_0\|_m +\|\nabla \underline u\|_m)
\\ & \hspace{1.85cm} \leq C \lambda_0^m\delta_0^{1/2}
\hspace{2.15cm} \mbox{ for all } \;\theta=1\ldots k,\; m\geq 1,
\\ & \|\nabla^{(m)}\langle \partial_{pq}^2u_0,
E_0^\theta\rangle\|_0\leq C\lambda_0^{m+1}\delta_0^{1/2} \quad \mbox{
  for all } \;\theta=1\ldots k, \;m\geq 0,
\end{split}
\end{equation}
where we used (\ref{ithnain}). Finally, writing:
$$\mathcal{D}(g_0-\delta_0H_0, u_0) = \mathcal{D}(g-\delta_0H_0, u)
*\phi_l - \big((\nabla u_0)^T\nabla u_0 - ((\nabla u)^T\nabla u)*\phi_l\big)$$ 
we estimate by \ref{stima1}, \ref{stima4}, \ref{Asu3}, for all
$m\geq 1$:
\begin{equation}\label{khamsa}
\begin{split}
& \|\nabla^{(m)}\Big(\mathcal{D}(g-\delta_0H_0, u)
*\phi_l \Big)\|_0\leq \frac{C}{l^m}\frac{r_0\delta_0}{4} = C \lambda_0^m\delta_0,
\\ & \|\nabla^{(m)}\big((\nabla u_0)^T\nabla u_0 - ((\nabla u)^T\nabla
u)*\phi_l\big)\|_0\leq C l^{2-m}\|\nabla^2u\|_0^2\leq C \lambda_0^m\delta_0,
\end{split}
\end{equation}
so that, in virtue of (\ref{uahid}) and \ref{Asu3} we get:
\begin{equation}\label{sitta}
\begin{split}
& \|\mathcal{D}(g_0-\delta_0H_0, u_0)
\|_0\leq \frac{r_0}{4}\delta_0 + \frac{r_0}{12}\delta_0= \frac{r_0}{3}\delta_0,
\\ & \|\nabla^{(m)}\mathcal{D}(g_0-\delta_0H_0, u_0)\|_0\leq 
C \lambda_0^m\delta_0 \quad \mbox{ for all } m\geq 1.
\end{split}
\end{equation}

\medskip

{\bf 2. (Setting up the induction)}  Starting with $u_0$, we will
inductively define the immersions:
$$\{u_i\in\mathcal{C}^1(\bar\omega, \R^{d+k})\}_{i=1}^N,\quad\mbox{
  where }\; N=lcm(d_*,k),$$
and their (sufficiently regular) normal frames $\{E_i^\theta\in\mathcal{C}(\bar\omega,
\R^{d+k})\}_{\theta=1}^k$, that is:
$$ (\nabla u_i)^TE_i^\theta=0, \quad |E_i^\theta|=1,\quad \langle
E_i^\theta, E_i^{\bar\theta}\rangle =0 \quad \mbox{in } \bar\omega \quad
\mbox{for all } \; \theta,\bar\theta=1\ldots k, \;
\theta\neq\bar\theta,\; i=1\ldots N.$$
The $N$-tuple of frequencies $\{\lambda_i\}_{i=1}^N$ is given by the formula:
$$\lambda_i = \lambda_{i-1}\cdot\left\{\begin{array}{ll}\sigma &
\mbox{if } k\mid i-1\\ \sigma^{1/2} & \mbox{if } d_*\mid i-1 \mbox{ and 
} i=2\ldots N\\ 1 & \mbox{otherwise.}
\end{array}\right.$$  
In particular, $\lambda_1=\lambda_0\sigma$ and, given $j=0\ldots J-1$,
$s=0\ldots S-1$ it follows that:
\begin{equation} \label{sabaa}
\lambda_i = \lambda_0\sigma^{1+j+s/2}\quad\mbox{ for all } i\in (jk,
(j+1)k]\cap (sd_*, (s+1)d_*]. 
\end{equation}
The consecutive defect measures $\{\delta_s\}_{s=1}^S$ are defined as
a decreasing $S$-tuple, through:
\begin{equation}\label{thamania}
\delta_s = \frac{\delta_0}{\sigma^s}\tilde C_s
\end{equation}  
with the help of an increasing $S$-tuple of constants $\tilde
C_1\ldots \tilde C_S$, all greater than $1$ and depending only on
$\omega, \underline u, \underline\gamma, d, k$.
In particular, setting $\tilde C_0=1$, we request that:
\begin{equation}\label{assu_CC}
\sigma\geq \frac{\tilde C_s}{\tilde C_{s-1}} \geq 1\quad\mbox{for all }\; s=1\ldots S,
\end{equation}
which can always be achieved by setting $\underline\sigma$ large enough.
To define the immersions $u_i$, we recall Lemma \ref{lem_step2} and the oscillatory
profiles $\Gamma$, $\bar\Gamma$, $\dbar\Gamma$ therein. First, for
$i=1\ldots N$ we uniquely write:
$$i=jk+\theta = sd_*+\rho \quad \mbox{ with } \; \theta=1\ldots k, \;
\rho = 1\ldots d_*$$
and then set, in an induction step from $i-1$ to $i$:
\begin{equation}\label{tisaa}
\begin{split}
& u_i = u_{i-1} + \frac{\Gamma(\lambda_i\langle x,\eta_\rho\rangle)}{\lambda_i}a_\rho^sE_{i-1}^\theta
+ \frac{\bar\Gamma(\lambda_i\langle x,\eta_\rho\rangle)}{\lambda_i}(a_\rho^s)^2T_{i-1}\eta_\rho
+ \frac{\dbar\Gamma(\lambda_i\langle
  x,\eta_\rho\rangle)}{\lambda_i^2}a_\rho^sT_{i-1}\nabla a_\rho^s
\\ & \mbox{where: } T_{i}=(\nabla u_i)\big((\nabla u_i)^T\nabla u_i\big)^{-1}.
\end{split}  
\end{equation}
The maps $a_\rho^s\in\mathcal{C}^1(\bar\omega)$ for $s=0\ldots S-1$,
$\rho=1\ldots d_*$ will be defined in Step 3 through an application of
Lemma \ref{lem_dec_def} to the defect
$\mathcal{D}(g_0-\delta_sH_0,u_{sd_*})$. The unit vector $\eta_\rho\in\R^d$ is
as in Lemma \ref{lem_dec_def}. In Steps 3-6 below we will prove that
for all $i=1\ldots N$ there holds: 
\begin{align*}
& \frac{1}{3\underline\gamma}\Id_d\leq (\nabla u_i)^T\nabla u_i\leq
3\underline\gamma\Id_d\quad\mbox{ in } \;\bar\omega,
\tag*{(\theequation)$_1$} \refstepcounter{equation} \label{Con1}\\
& \|u_i-u_{i-1}\|_1\leq C\delta_0^{1/2}, \tag*{(\theequation)$_2$} \label{Con2}\\
& \|\nabla^{(m)}\nabla^{(2)}u_i\|_0+ \|\nabla^{(m)}\nabla E_i^\theta\|_0
\leq C\lambda_0\lambda_i^m\sigma^j\delta_0^{1/2}\\
& \hspace{2.6cm} \mbox{ for all }(j-1)k<i\leq jk, \; j=1\ldots J,
 \; \theta=1\ldots k,\; m\geq 0,
\tag*{(\theequation)$_3$} \label{Con_full}\\
&\hspace{-3mm}\left. \begin{array}{ll}\|\nabla^{(m)}\langle\partial_{pq}^2u_i, E_i^\theta\rangle\|_0\leq
C\lambda_0\lambda_i^m\sigma^j\delta_0^{1/2} 
 & \mbox{ for } \; jk\leq i<jk+\theta \\
\|\nabla^{(m)}\langle\partial_{pq}^2u_i, E_i^\theta\rangle\|_0\leq
C\lambda_0\lambda_i^m\sigma^{j+1}\delta_0^{1/2} 
& \mbox{ for } \; jk + \theta\leq i\leq (j+1)k \end{array}\right\}
 \\ & \hspace{2.6cm} \mbox{ for all }\; j=0\ldots J-1,\; \theta =1\ldots k,\; m\geq 0,
\tag*{(\theequation)$_4$} \label{Con3}\\     
& \|\nabla^{(m)}\mathcal{D}(g_0-\delta_sH_0, u_{sd_*})\|_0\leq \frac{r_0}{5}\lambda_{sd_*}^m\delta_s
\quad\mbox{ for all } s=1\ldots S, \; m\geq 0.
\tag*{(\theequation)$_5$} \label{Con5}
\end{align*}
Observe that, writing the first assertion in \ref{Con_full} as:
$$\|\nabla^{(m)}\nabla^{(2)}u_i\|_0\leq C\lambda_i^{m+1}A \quad\mbox{
  with }\; A=\frac{\lambda_0}{\lambda_i}\sigma^j\delta_0^{1/2}$$
and noting that $A\leq \sigma^J\delta_0^{1/2}<1$ by the last
assumption in (\ref{Asu1}), we use Lemma \ref{lem_Tu} to get:
\begin{equation}\label{Con6}
\begin{split}
\|T_i\|_0\leq C, \quad & \|\nabla^{(m)}T_i\|_0\leq C\lambda_0\lambda_i^{m-1}\sigma^j\delta_0^{1/2}
\\ & \mbox{for all } \; (j-1)k<i\leq jk,\; j=1\ldots J,\; m\geq 1.
\end{split}
\end{equation}
Also, we note the simplified version of the bound above and
the second bound in \ref{Con_full}:
\begin{equation}\label{Con7}
\|\nabla^{(m)}E_i^\theta\|_0+\|\nabla^{(m)}T_i\|_0\leq
C\lambda_i^{m}\quad \mbox{for all } \; m\geq 0,
\end{equation}
which clearly holds at $m=0$, whereas at $m\geq 1$ we use that 
$\sigma^j\delta_0^{1/2}\leq 1$ in virtue of the last assumption in (\ref{Asu1}).
The statements \ref{Con1}-\ref{Con5} will be proved inductively.
For the induction base, observe that at $i=0$, the assertion \ref{Con1}
already holds by (\ref{thalata}), while \ref{Con_full} holds by the
last claim in (\ref{ithnain}) and the first one in
(\ref{arbaa}). Whereas, the second assertion in there validates the first
inequality in \ref{Con3}, meanwhile the second inequality is void at $i=0$.
Finally, the estimate \ref{Con5} holds for $s=0$ in the form of (\ref{sitta}).

\medskip

{\bf 3. (The defect decomposition)}  Assume \ref{Con5} at some
$s=1\ldots S-1$ or assume (\ref{sitta}) if $s=0$. We want to apply
Lemma \ref{lem_dec_def} to the scaled defect:
$$H=\frac{1}{\delta_s}\mathcal{D}(g_0-\delta_{s+1}H_0, u_{sd_*}).$$
To check that $H(x)\in B(H_0, r_0)$ for all $x\in\bar\omega$, we note
that $H-H_0 = \big(\mathcal{D}(g_0-\delta_sH_0,
u_{sd_*})-\delta_{s+1}H_0\big)/\delta_s$ which implies:
\begin{equation*}
\| H-H_0\|_0\leq \frac{\|\mathcal{D}(g_0-\delta_sH_0,
  u_{sd_*})\|_0}{\delta_s} + \frac{\delta_{s+1}}{\delta_s}H_0 \leq \frac{r_0}{3}
+ \frac{\tilde C_{s+1}}{\tilde C_s\sigma}|H_0|< \frac{r_0}{2}
\end{equation*}  
where the first term was estimated by \ref{Con5} or (\ref{sitta}), and
for the second term is bounded in view of the following assumption,
complementary to that of (\ref{assu_CC}):
\begin{equation}\label{assu_CC2}
\frac{\tilde C_s}{\tilde C_{s-1}} \cdot
\frac{6|H_0|}{r_0}<\underline\sigma \quad\mbox{for all }\; s=1\ldots S.
\end{equation}
Therefore, by Lemma \ref{lem_dec_def} we obtain the decomposition:
\begin{equation*}
\begin{split}
& H=\sum_{\rho=1}^{d_*}\bar a_\rho(H) \eta_\rho\otimes\eta_\rho
\quad\mbox{ with }\; \|\bar a_\rho(H)-1\|_0\leq \frac{1}{2} \;\mbox{ for all }\;\rho=1\ldots d_*,
\\ & \|\nabla^{(m)}\bar a_\rho(H)\|_0\leq C\|\nabla^{(m)}H\|_0\leq 
C \frac{\|\nabla^{(m)}\mathcal{D}(g_0-\delta_sH_0, u_{sd_*})\|_0}{\delta_s} \leq
C\lambda_{sd_*}^m \quad\mbox{ for all }\;m\geq 1.
\end{split}
\end{equation*}
Note that the above bounds are independent of $\tilde C_{s+1}$
or $\delta_{s+1}$ and only utilize properties of quantities
obtained in the so-far inductive process, as long as (\ref{assu_CC2}) holds. We now define:
\begin{equation}\label{aszra}
(a_\rho^s)^2=\delta_s \bar a_\rho(H), \quad\mbox{ so that: }\;
\mathcal{D}(g_0-\delta_{s+1}H_0, u_{sd_*})=\sum_{\rho=1}^{d_*}(a_\rho^s)^2\eta_\rho\otimes\eta_\rho,
\end{equation}
and observe the following:
\begin{equation}\label{ahadashar}
\begin{split}
&\|(a_\rho^s)^2-\delta_s\|_0\leq \frac{\delta_s}{2} \quad\mbox{ and } \quad
\|a_\rho^s - \delta_s^{1/2}\|_0= \big\|\frac{(a_\rho^s)^2-\delta_s}{a_\rho^s+\delta_s^{1/2}}\big\|_0
\leq \frac{\delta_s/2}{\delta_s^{1/2}}=\frac{\delta_s^{1/2}}{2},\\
& \|\nabla^{(m)}(a_\rho^s)^2\|_0\leq C\lambda_{sd_*}^m\delta_s 
\quad\mbox{ and } \quad \|\nabla^{(m)}a_\rho^s\|_0\leq C\lambda_{sd_*}^m\delta_s^{1/2}
\quad \mbox{ for all }\; m\geq 1,
\end{split}    
\end{equation}
where the last bound follows by the Fa\'a di Bruno inequality in:
\begin{equation*}
\begin{split}
& \|\nabla^{(m)}a_\rho^s\|_0 \leq C\Big\|\sum_{p_1+2p_2+\ldots
  +mp_m=m} \hspace{-6mm}
(a_\rho^s)^{2(\frac{1}{2}-p_1-\ldots -p_m)}\prod_{t=1}^m|\nabla^{(t)}(a_\rho^s)^2|^{p_t}\Big\|_0
\\ & \leq C\|a_\rho^s\|_0\hspace{-3mm}\sum_{p_1+2p_2+\ldots +mp_m=m}\;
\prod_{t=1}^m\Big\|\frac{\nabla^{(t)}(a_\rho^s)^2}{(a_\rho^s)^2}\Big\|_0^{p_t}\leq
C\delta_s^{1/2}\hspace{-3mm}\sum_{p_1+2p_2+\ldots +mp_m=m}\;
\prod_{t=1}^m\Big\|\frac{\delta_s\lambda_{sd_*}^t}{\delta_s/2}\Big\|_0^{p_t}
\\ & \leq C\lambda_{sd_*}^m\delta_s^{1/2},
\end{split}    
\end{equation*}
Again, we point out that constants in (\ref{ahadashar}) depend only on
$\omega, \underline u, \underline\gamma, d, k$ through quantities
obtained in the so-far inductive process, but are independent of $\tilde C_{s+1}$ or $\delta_{s+1}$.

\medskip

{\bf 4. (Proof of \ref{Con1}, \ref{Con2} and \ref{Con_full})} 
Fix $i=1\ldots N$. We will prove the indicated assertions, assuming that
they hold up to the counter $i-1$. We write:
\begin{equation}\label{ahadashar1}
  i=jk+\theta = sd_*+\rho \quad \mbox{ with } \; \theta=1\ldots k, \;
\rho = 1\ldots d_*,\; j=0\ldots J-1,\; s=0\ldots S-1,
\end{equation}
and apply (\ref{tisaa}), (\ref{ahadashar}), and \ref{Con_full} 
(\ref{Con6}) in the form $\|\nabla^{(v)}E^\theta_{i-1}\|_0+
\|\nabla^{(v)}T_{i-1}\|_0\leq C\lambda_{i-1}^v$, valid in view of the last
assumption in (\ref{Asu1}), to estimate:
\begin{equation}\label{ithnashar}
\begin{split}
& \|\nabla^{(m)}(u_i-u_{i-1})\|_0\leq C\sum_{t+w+v=m} \Big(\lambda_i^{t-1}
\|\nabla^{(w)}a_\rho^s\|_0\|\nabla^{(v)}E_{i-1}^\theta\|_0 \\ &
\hspace{3cm} + \lambda_i^{t-1}
\|\nabla^{(w)}(a_\rho^s)^2\|_0\|\nabla^{(v)}T_{i-1}^\theta\|_0 
+ \lambda_i^{t-2} \|\nabla^{(w+1)}(a_\rho^s)^2\|_0\|\nabla^{(v)}T_{i-1}^\theta\|_0 \Big)
\\ & \leq C\sum_{t+w+v=m} \Big(\lambda_i^{t-1}
\delta_s^{1/2}\lambda_{sd_*}^w\lambda_{i-1}^v + \lambda_i^{t-1}
\delta_s\lambda_{sd_*}^w\lambda_{i-1}^v 
+ \lambda_i^{t-2} \delta_s\lambda_{sd_*}^{w+1}\lambda_{i-1}^v \Big)
\\ & \leq C\lambda_i^{m-1}\delta_s^{1/2} \quad\mbox{ for all }\; m\geq 0,
\end{split}
\end{equation}
since both $\lambda_{sd_*}$ and $\lambda_{i-1}$ are not greater than
$\lambda_i$. In particular, there follows \ref{Con2}, namely:
$$\|u_i-u_{i-1}\|_1\leq C\delta_s^{1/2}\leq C\delta_0^{1/2}.$$
Using now the induction assumption and (\ref{ithnain}), we obtain:
$$\|u_i-u\|_1\leq \|u_i-u_{0}\|_1 +\|u_0-u\|_1\leq C\delta_0^{1/2},$$
which yields \ref{Con1} upon setting $\underline\delta$ sufficiently
small to get, when $\delta_0=\delta\leq\underline\delta$:
\begin{equation*}
\begin{split}
\|(\nabla u_i)^T\nabla u_i - (\nabla u)^T\nabla u\|_0 &\leq \|\nabla
(u_i-u)\|_0\big(2\|\nabla u\|_0 + \|\nabla(u_i-u)\|_0\big)\\ & \leq
C\delta_0^{1/2}(2(2\underline\gamma d)^{1/2}+1)\leq C\underline\delta^{1/2}
\leq \frac{1}{6\underline\gamma},
\end{split}
\end{equation*}
where we use Lemma \ref{lem_det} and recall \ref{Asu2}.
The first inequality in \ref{Con_full} follows by (\ref{ithnashar}) in:
\begin{equation*}
\|\nabla^{(m)}\nabla^{(2)}(u_i-u_{i-1})\|_0 \leq
C\lambda_i^{m+1}\delta_s^{1/2} \leq C \lambda_0\lambda_i^m\sigma^{j+1}\delta_0^{1/2}
\quad\mbox{ for all }\; m\geq 0,
\end{equation*}
because by (\ref{sabaa}) and (\ref{thamania}) there holds:
\begin{equation}\label{ithnashar1}
\lambda_i^{m+1}\delta_s^{1/2}=
\lambda_0\lambda_i^m \delta_0^{1/2} \frac{\lambda_i}{\lambda_0}\big(\frac{\delta_s}{\delta_0}\big)^{1/2}
= \lambda_0\lambda_i^m \delta_0^{1/2}\sigma^{1+j+s/2}\frac{{\tilde
    C_s}^{1/2}}{\sigma^{s/2}}\leq C \lambda_0\lambda_i^m\sigma^{j+1}\delta_0^{1/2}. 
\end{equation}
To define the normal frame $\{E_i^\theta\doteq
E_{u_i}^\theta\}_{\theta=1}^k$ and prove the second inequality in
\ref{Con_full}, we use Lemma \ref{lem_propa} (iii). We first validate its assumptions. By
(\ref{ithnashar}) and the induction assumption:
\begin{equation*}
\begin{split}
& \|\nabla^{(m)}(\nabla u_{i}- \nabla u_{i-1})\|_0\leq
C\lambda_i^m\delta_s^{1/2} \qquad\mbox{for all }\; m\geq 0,
\\ & \|\nabla^{(m)}\nabla^{(2)}u_{i-1}\|_0
+ \|\nabla^{(m+1)}E_{i-1}^\theta\|_0\leq C\lambda_{i-1}^{m+1}\leq C\lambda_i^{m+1}
\quad \mbox{for all }\; \theta=1\ldots k,\; m\geq 0.
\end{split}
\end{equation*}
Thus, taking $A=\delta_s^{1/2}$ and $\mu=\lambda_i$, the estimate in
Lemma \ref{lem_propa} (iii) becomes:
\begin{equation}\label{thalatatashar}
\|\nabla^{(m)}(E^\theta_{i}-E^\theta_{i-1})\|_0\leq C\lambda_i^m\delta_s^{1/2}
\quad \mbox{for all }\; \theta=1\ldots k,\; m\geq 0.
\end{equation}
This completes the proof of \ref{Con_full}, because
$\lambda_i^{m+1}\delta_s^{1/2}\leq \lambda_0\lambda_i^m\sigma^{j+1}\delta_0^{1/2}$,
as shown in (\ref{ithnashar1}).

\medskip

{\bf 5. (Proof of \ref{Con3})}  We fix the index $i=1\ldots N$ and express it as
in (\ref{ahadashar1}), with the appropriate $\theta$ and $\rho$. Now,
fix any $\bar\theta=1\ldots k$ and estimate derivatives of the three terms in: 
\begin{equation*}
\begin{split}
\langle \partial_{pq}^2u_i, E_i^{\bar\theta}\rangle = I+ II+ III,
\quad\mbox{where }\;
& I =\langle \partial_{pq}^2u_{i-1}, E_{i-1}^{\bar\theta}\rangle,\quad
II= \langle \partial_{pq}^2u_{i}, E_i^{\bar\theta} - E_{i-1}^{\bar\theta}\rangle,\\ &
III=\langle \partial_{pq}^2(u_i-u_{i-1}), E_{i-1}^{\bar\theta}\rangle.
\end{split}
\end{equation*}
For the first term $I$, note that $jk\leq i-1<(j+1)k$, so the
induction assumption on \ref{Con3} implies that for any $\bar\theta$ there holds:
$$\|\nabla^{(m)}I\|_0\leq C\lambda_0\lambda_{i-1}^m\sigma^{j+1}\delta_0^{1/2}\leq 
C\lambda_0\lambda_{i}^m\sigma^{j+1}\delta_0^{1/2} \quad\mbox{for all }\; m\geq 0.$$
On the other hand, when $jk<i=jk+\theta<jk+\bar\theta$ then, trivially 
$jk\leq i-1<jk+\bar\theta$, so:
$$\|\nabla^{(m)}I\|_0\leq C\lambda_0\lambda_{i-1}^m\sigma^{j}\delta_0^{1/2}\leq 
C\lambda_0\lambda_{i}^m\sigma^{j}\delta_0^{1/2} \quad\mbox{for all }\; m\geq 0.$$
For the terms $II$, we use the already proven bound on
$\|\nabla^{(m+2)}u_i\|_0$ and (\ref{thalatatashar}) to get:
\begin{equation*}
\begin{split}
\|\nabla^{(m)}II\|_0&\leq C\sum_{t+v=m}\|\nabla^{(t+2)}u_i\|_0\|\nabla^{(v)}
(E_i^{\bar\theta} - E_{i-1}^{\bar\theta})\|_0\\ & \leq C\sum_{t+v=m}
\lambda_0\lambda_i^t\sigma^{j+1}\delta_0^{1/2}\lambda_i^v\delta_s^{1/2}
\leq C \lambda_0\lambda_i^m\sigma^{j}\delta_0^{1/2}(\sigma\delta_s^{1/2}) \\
& \leq C \lambda_0\lambda_i^m\sigma^{j}\delta_0^{1/2} \quad\mbox{for all }\; m\geq 0,
\end{split}
\end{equation*}
as $\sigma\delta_s^{1/2}\leq \sigma\delta_0^{1/2}\leq 1$ by the
last assumption in (\ref{Asu1}).
For the term $III$, we use (\ref{tisaa}) to get:
\begin{equation*}
\begin{split}
III= & \, A+B+D+E \quad\mbox{where }\;
A = \partial_{pq}^2\Big(\frac{\Gamma(\lambda_i\langle x,
  \eta_\rho\rangle)}{\lambda_i}a_\rho^s\Big) \langle E_{i-1}^\theta, E_{i-1}^{\bar\theta}\rangle,
\\ & B = \partial_{p}\Big(\frac{\Gamma(\lambda_i\langle x,
  \eta_\rho\rangle)}{\lambda_i}a_\rho^s\Big) \langle \partial_qE_{i-1}^\theta, E_{i-1}^{\bar\theta}\rangle
+ \partial_{q}\Big(\frac{\Gamma(\lambda_i\langle x,
  \eta_\rho\rangle)}{\lambda_i}a_\rho^s\Big) \langle \partial_pE_{i-1}^\theta, E_{i-1}^{\bar\theta}\rangle
\\ & \hspace{0.8cm} + \frac{\Gamma(\lambda_i\langle x,
  \eta_\rho\rangle)}{\lambda_i}a_\rho^s\langle \partial_{pq}E_{i-1}^\theta, E_{i-1}^{\bar\theta}\rangle,
\\ &  D =  \partial_{p}\Big(\frac{\bar\Gamma(\lambda_i\langle x,
  (\eta_\rho)^2\rangle)}{\lambda_i}(a_\rho^s)^2\Big) \langle (\partial_qT_{i-1})\eta_\rho, E_{i-1}^{\bar\theta}\rangle
+ \partial_{q}\Big(\frac{\bar\Gamma(\lambda_i\langle x,
  \eta_\rho\rangle)}{\lambda_i}(a_\rho^s)^2\Big) \langle (\partial_pT_{i-1})\eta_\rho, E_{i-1}^{\bar\theta}\rangle
\\ & \hspace{0.8cm} + \frac{\bar\Gamma(\lambda_i\langle x,
  \eta_\rho\rangle)}{\lambda_i}(a_\rho^s)^2\langle (\partial^2_{pq}T_{i-1})\eta_\rho, E_{i-1}^{\bar\theta}\rangle,
\\ &  E =  \Big\langle (\partial_pT_{i-1})\partial_q \Big(\frac{\dbar\Gamma(\lambda_i\langle x,
  \eta_\rho\rangle)}{2\lambda_i^2}\nabla(a_\rho^s)^2\Big), E_{i-1}^{\bar\theta}\Big\rangle  
\\ & \hspace{0.8cm} + \Big\langle (\partial_qT_{i-1})\partial_p \Big(\frac{\dbar\Gamma(\lambda_i\langle x,
  \eta_\rho\rangle)}{2\lambda_i^2}\nabla(a_\rho^s)^2\Big), E_{i-1}^{\bar\theta}\Big\rangle  
+ \frac{\dbar\Gamma(\lambda_i\langle x, \eta_\rho\rangle)}{2\lambda_i^2}
\langle (\partial^2_{pq}T_{i-1})\nabla(a_\rho^s)^2, E_{i-1}^{\bar\theta}\rangle. 
\end{split}
\end{equation*}
Note that the term $A$ is nonzero only when $\theta=\bar\theta$, 
in which case we get by (\ref{ahadashar}) and (\ref{ithnashar1}):
\begin{equation*}
\begin{split}
\|\nabla^{(m)}A\|_0 & \leq
C\hspace{-2mm}\sum_{t+v=m+2}\hspace{-2mm}\lambda_i^{t-1}\|\nabla^{(v)}a_\rho^s\|_0\leq 
C\hspace{-2mm}\sum_{t+v=m+2}\hspace{-2mm}\lambda_i^{t-1}\lambda_{sd_*}^v\delta_s^{1/2}
\\ & \leq C \lambda_i^{m+1}\delta_s^{1/2}\leq
C\lambda_0\lambda_i^m\sigma^{j+1}\delta_0^{1/2}\quad\mbox{ for all }\; m\geq 0.
\end{split}
\end{equation*}
When $jk<i=jk+\theta<jk+\bar\theta$, then certainly $\theta\neq
\bar\theta$, so the above term is not present in $III$. We now show
that all other terms estimate by
$C\lambda_0\lambda_i^m\sigma^{j}\delta_0^{1/2}$ in their $m$-th
derivatives, for any $\bar\theta=1\ldots k$. Indeed, from
(\ref{ahadashar}) and \ref{Con_full}, (\ref{Con7}) at $i-1$ we get:
\begin{equation*}
\begin{split}
\|\nabla^{(m)}B\|_0  & \leq C\hspace{-2mm}\sum_{t+w+v=m}\;\sum_{\bar t+\bar w =
  t+1}\hspace{-2mm}\lambda_i^{\bar t-1}\|\nabla^{(\bar w)}a_\rho^s\|_0
\|\nabla^{(w+1)}E_{i-1}^{\theta}\|_0 \|\nabla^{(v)}E_{i-1}^{\bar\theta}\|_0 
\\ & \quad + C\hspace{-2mm}\sum_{t+w+v+\bar v=m}\hspace{-2mm}
\lambda_i^{t-1}\|\nabla^{(w)}a_\rho^s\|_0
\|\nabla^{(v+2)}E_{i-1}^{\theta}\|_0 \|\nabla^{(\bar v)}E_{i-1}^{\bar\theta}\|_0 
\\ & \leq C\hspace{-2mm}\sum_{t+v+w=m}\;\sum_{\bar t+\bar w =
  t+1}\hspace{-2mm}\lambda_i^{\bar t-1}\lambda_{sd_*}^{\bar w}\delta_s^{1/2}
\lambda_0\lambda_{i-1}^w\sigma^{j+1}\delta_0^{1/2}\lambda_{i-1}^v 
\\ & \quad + C\hspace{-2mm}\sum_{t+w+v+\bar v=m}\hspace{-2mm}
\lambda_i^{t-1}\lambda_{sd_*}^{w}\delta_s^{1/2}
\lambda_0\lambda_{i-1}^{v+1}\sigma^{j+1}\delta_0^{1/2}\lambda_{i-1}^{\bar v} 
\\ & \leq C \lambda_0\lambda_i^{m}\sigma^j\delta_0^{1/2}
(\sigma\delta_s^{1/2}) \leq C\lambda_0\lambda_i^m\sigma^{j}\delta_0^{1/2}
\quad\mbox{ for all }\; m\geq 0,
\end{split}
\end{equation*}
since $\sigma\delta_s^{1/2}\leq 1$. Similarly, replacing the
application of \ref{Con_full} by (\ref{Con6}), implies:
\begin{equation*}
\begin{split}
\|\nabla^{(m)}D\|_0  & \leq C\hspace{-2mm}\sum_{t+w+v=m}\;\sum_{\bar t+\bar w =
  t+1}\hspace{-2mm}\lambda_i^{\bar t-1}\|\nabla^{(\bar w)}(a_\rho^s)^2\|_0
\|\nabla^{(w+1)}T_{i-1}\|_0 \|\nabla^{(v)}E_{i-1}^{\bar\theta}\|_0 
\\ & \quad + C\hspace{-2mm}\sum_{t+w+v+\bar v=m}\hspace{-2mm}
\lambda_i^{t-1}\|\nabla^{(w)}(a_\rho^s)^2\|_0
\|\nabla^{(v+2)}T_{i-1}\|_0 \|\nabla^{(\bar v)}E_{i-1}^{\bar\theta}\|_0 
\\ & \leq C \lambda_0\lambda_i^{m}\sigma^j\delta_0^{1/2}
(\sigma\delta_s) \leq C\lambda_0\lambda_i^m\sigma^{j}\delta_0^{1/2}
\quad\mbox{ for all }\; m\geq 0.
\end{split}
\end{equation*}
Finally, for the term $E$ we get, as above:
\begin{equation*}
\begin{split}
\|\nabla^{(m)}E\|_0  & \leq C\hspace{-2mm}\sum_{t+w+v=m}\;\sum_{\bar t+\bar w =
  t+1}\hspace{-2mm}\lambda_i^{\bar t-2}\|\nabla^{(\bar w+1)}(a_\rho^s)^2\|_0
\|\nabla^{(w+1)}T_{i-1}\|_0 \|\nabla^{(v)}E_{i-1}^{\bar\theta}\|_0 
\\ & \quad + C\hspace{-2mm}\sum_{t+w+v+\bar v=m}\hspace{-2mm}
\lambda_i^{t-2}\|\nabla^{(w+1)}(a_\rho^s)^2\|_0
\|\nabla^{(v+2)}T_{i-1}\|_0 \|\nabla^{(\bar v)}E_{i-1}^{\bar\theta}\|_0 
\\ &  \leq C\hspace{-2mm}\sum_{t+v+w=m}\;\sum_{\bar t+\bar w =
  t+1}\hspace{-2mm}\lambda_i^{\bar t-2}\lambda_{sd_*}^{\bar w+1}\delta_s
\lambda_0\lambda_{i-1}^w\sigma^{j+1}\delta_0^{1/2}\lambda_{i-1}^v 
\\ & \quad + C\hspace{-2mm}\sum_{t+w+v+\bar v=m}\hspace{-2mm}
\lambda_i^{t-2}\lambda_{sd_*}^{w+1}\delta_s
\lambda_0\lambda_{i-1}^{v+1}\sigma^{j+1}\delta_0^{1/2}\lambda_{i-1}^{\bar v} 
\\ & \leq C \lambda_0\lambda_i^{m}\sigma^j\delta_0^{1/2}
(\sigma\delta_s) \leq C\lambda_0\lambda_i^m\sigma^{j}\delta_0^{1/2}
\quad\mbox{ for all }\; m\geq 0.
\end{split}
\end{equation*}
This ends the proof of \ref{Con3} at the counter value $i$.

\medskip

{\bf 6. (Proof of \ref{Con5})} Fix the index $s=0\ldots S-1$ and assume
that \ref{Con1} - \ref{Con3} hold for all $i=1\ldots (s+1)d_*$,
together with \ref{Con5} if $s\geq 1$ or (\ref{sitta}) in case
$s=0$. To prove \ref{Con5} at the counter value $s+1$, we use the
decomposition (\ref{aszra}) and write:
\begin{equation}\label{arbatashar}
\begin{split}
& \mathcal{D}(g_0 - \delta_{s+1}H_0, u_{(s+1)d_*}) \\ & = \mathcal{D}(g_0 - \delta_{s+1}H_0, u_{sd_*}) 
-\Big( (\nabla u_{(s+1)d_*})^T \nabla u_{(s+1)d_*} - (\nabla u_{sd_*})^T \nabla u_{sd_*}\Big)
= \sum_{\rho=1}^{d_*} I_\rho, \\ & \mbox{where: }\;
I_\rho= (a_\rho^s)^2\eta_\rho\otimes\eta_\rho - 
\Big( (\nabla u_{sd_*+\rho})^T \nabla u_{sd_*+\rho} - (\nabla
u_{sd_*+\rho-1})^T \nabla u_{sd_*+\rho-1}\Big).
\end{split}
\end{equation}
Recalling (\ref{sisi}) in Lemma \ref{lem_step2}, we see that each term
$I_\rho$ for $\rho=1\ldots d_*$
has the following form, where we denote $i=sd_*+\rho = jk+\theta$
for the uniquely defined $j=0\ldots J-1$ and $\theta=1\ldots k$:
\begin{equation*}
\begin{split}
& I_\rho = A+B+D-\mathcal{R}, \quad \mbox{where: }\; A = \frac{2\Gamma(\lambda_i\langle x,
  \eta_\rho\rangle)}{\lambda_i}a_\rho^s\big[\langle \partial_{pq}^2u_{i-1},
E_{i-1}^\theta\rangle\big]_{p,q=1\ldots d} \\ & B =
-\frac{2\dbar\Gamma(\lambda_i\langle
  x,\eta_\rho\rangle)}{\lambda_i^2}a_\rho^s\nabla^2a_\rho^s,
\qquad D = -\frac{\tbar\Gamma(\lambda_i\langle
  x,\eta_\rho\rangle)}{\lambda_i^2}\nabla a_\rho^s\otimes \nabla a_\rho^s,
\end{split}
\end{equation*}
and where $\mathcal{R}$ comprises the terms $\mathcal{R}_0\ldots\mathcal{R}_4$
given in Lemma \ref{lem_step2} with the following defining quantites: $\lambda=\lambda_i$,
$a=a_\rho^s$, $\eta=\eta_\rho$, $u=u_{i-1}$, $E_u=E_{i-1}^\theta$. We now estimate,
in virtue of (\ref{ahadashar}), and of \ref{Con3} used with $jk\leq i-1<jk+\theta$:
\begin{equation*}
\begin{split}
\|\nabla^{(m)}A\|_0 & \leq C\hspace{-2mm}\sum_{t+w+v=m}\hspace{-2mm}
\lambda_i^{t-1}\|\nabla^{(w)}a_\rho^s\|_0\sum_{p,q=1\ldots d}\hspace{-2mm}
\|\nabla^{(v)}\langle \partial_{pq}^2u_{i-1}, E_{i-1}^\theta\rangle\|_0
\\ & \leq C \hspace{-2mm}\sum_{t+w+v=m}\hspace{-2mm}
\lambda_i^{t-1}\lambda_{sd_*}^w\delta_s^{1/2} \lambda_0\lambda_{i-1}^v\sigma^j\delta_0^{1/2}
\leq C \lambda_0\lambda_i^{m-1}\delta_s^{1/2}\sigma^j\delta_0^{1/2} 
\\ & = C\lambda_i^m\delta_s\sigma^j
\frac{\lambda_0}{\lambda_i}\big(\frac{\delta_0}{\delta_s}\big)^{1/2}\leq
C\lambda_{(s+1)d_*}^m\frac{\delta_s}{\sigma} \quad\mbox{ for all }\; m\geq 0.
\end{split}
\end{equation*}
The last bound is due to $\lambda_i\leq \lambda_{(s+1)d_*}$ and to
(\ref{sabaa}), (\ref{thamania}) in:
$$\sigma^j \frac{\lambda_0}{\lambda_i}\big(\frac{\delta_0}{\delta_s}\big)^{1/2}
= \frac{\sigma^j}{\sigma^{1+j+s/2}} \frac{\sigma^{s/2}}{\tilde
C_s^{1/2}}\leq \frac{1}{\sigma}.$$
The next two terms $B$ and $D$ are similarly estimated by:
\begin{equation*}
\begin{split}
\|\nabla^{(m)}(B+D)\|_0 & \leq C\hspace{-2mm}\sum_{t+w+v=m}\hspace{-2mm}
\lambda_i^{t-2} \big(\|\nabla^{(w)}a_\rho^s\|_0\|\nabla^{(v+2)}a_\rho^s\|_0 +
\|\nabla^{(w+1)} a_\rho^s\|_0\|\nabla^{(v+1)} a_\rho^s\|_0\big) 
\\ & \leq C \hspace{-2mm}\sum_{t+w+v=m}\hspace{-2mm}
\lambda_i^{t-2}\lambda_{sd_*}^{w+v+2}\delta_s
\leq C \lambda_i^{m}\frac{\delta_s}{(\lambda_i/\lambda_{sd_*})^2} \\ &
\leq C\lambda_{(s+1)d_*}^m\frac{\delta_s}{\sigma} \quad\mbox{ for all }\; m\geq 0,
\end{split}
\end{equation*}
because $\lambda_i/\lambda_{sd_*}\geq \sigma^{1/2}$ from the
definition of the progression of frequencies leading to (\ref{sabaa}).
We proceed to estimating derivatives of the remainder terms in
$\mathcal{R}$. For $\mathcal{R}_0$, we observe:
\begin{equation*}
\begin{split}
\|\nabla^{(m)}\mathcal{R}_0\|_0 & \leq C\hspace{-2mm}\sum_{t+w+v=m}\hspace{-2mm}
\lambda_i^{t} \big(\|\nabla^{(w)}(a_\rho^s)^4\|_0\|\nabla^{(v)}|T_{i-1}|^2\|_0
\leq C \hspace{-2mm}\sum_{t+w+v=m}\hspace{-2mm}
\lambda_i^{t}\lambda_{sd_*}^{w}\delta_s^2\lambda_{i-1}^v \\ & \leq C \lambda_i^{m}\delta_s^2
\leq C\lambda_{(s+1)d_*}^m\frac{\delta_s}{\sigma} \quad\mbox{ for all }\; m\geq 0,
\end{split}
\end{equation*}
where we used (\ref{ahadashar}), (\ref{Con7}) and the fact that
$\sigma \delta_s\leq \sigma\delta_0\leq 1$, evident from the last
of the conditions listed in (\ref{Asu1}). For the remaining terms constituting
$\mathcal{R}_1\ldots \mathcal{R}_4$, we note that each of them contains
the product of at least two quantities involving $a_\rho^s$ (or
$\nabla a_\rho^s$), which jointly contribute the multiplier $\delta_s$ in the estimate;
and the products of either another instance of $a_\rho^s$ or one of:
$\nabla^2u$ or $\nabla E_{i-1}^\theta$, each of which contribute the multiplier
bounded from above by $\delta_0^{1/2}$. By assuring $\delta_0^{1/2}$
to be sufficiently small with respect to the powers of $\sigma$
accumulated in the estimate of each term, through the last assumption
in (\ref{Asu1}), we arrive at:
\begin{equation*}
\|\nabla^{(m)}\mathcal{R}\|_0 \leq
C\lambda_{(s+1)d_*}^m\frac{\delta_s}{\sigma} \quad\mbox{ for all }\; m\geq 0,
\end{equation*}
Consequently, it follows that each $\nabla^{(m)}I_\rho$, for
$\rho=1\ldots d_*$, obeys the same bound as above, hence by
(\ref{arbatashar}) we may conclude that:
\begin{equation*}
\begin{split}  
& \|\nabla^{(m)}\mathcal{D}(g_0-\delta_{s+1}H_0, u_{(s+1)d_*})\|_0 \leq
C\lambda_{(s+1)d_*}^m\frac{\delta_s}{\sigma}
\\ & = \frac{r_0}{5}\lambda_{(s+1)d_*}^m\delta_{s+1} \cdot
\frac{5C}{r_0}\frac{\tilde C_s}{\tilde C_{s+1}}
\leq \frac{r_0}{5}\lambda_{(s+1)d_*}^m\delta_{s+1}
\quad\mbox{ for all }\; m\geq 0.
\end{split}
\end{equation*}
The last estimate is valid in view of (\ref{thamania}), provided that we define $\tilde C_{s+1}$ so that:
$$\tilde C_{s+1}\geq \frac{5C}{r_0}\tilde C_s.$$
This ends the proof of \ref{Con5} at the counter value $s+1$, and the
proof of all the inductive estimates \ref{Con1} - \ref{Con5}.

\medskip

{\bf 7. (The reparametrization and the end of proof)}  
We now define:
$$\tilde u = u_N,$$
whereupon \ref{Con1} - \ref{Con5} and (\ref{ithnain}), (\ref{thamania}) imply:
\begin{equation}\label{Conny}
\begin{split}
& \|u-\tilde u\|_1\leq C\delta_0^{1/2}, \qquad
\|\nabla^2 \tilde u\|_0\leq C\lambda_0\sigma^J\delta_0^{1/2},
\\ & \Big\|\mathcal{D}\Big(g - \frac{\tilde C_S\delta_0}{\sigma^S} H_0, \tilde u\Big)\Big\|_0
\leq \frac{r_0}{5}\frac{\tilde C_S\delta_0}{\sigma^S} + \frac{\|g\|_{r,\beta}}{\lambda_0^{r+\beta}}.
\end{split}
\end{equation}
These are essentially the claimed bounds \ref{Res1}, \ref{Res2}, up to
the constant $\tilde C_S$, that depends only on $\omega, \underline u,
\underline \gamma, d, k$. Given $\sigma$, we define the associated reparametrized relative frequency:
$$\sigma_{new} = \frac{\sigma}{\tilde C_S^{1/S}},$$
whereupon the first two bounds in (\ref{Conny}) remain the same and the last one
becomes exactly \ref{Res2}. We claim that the assumptions in
\ref{Ass1} for $\sigma_{new}$ imply assumptions in (\ref{Asu1}) for $\sigma$:
\begin{equation*}
\begin{split}
& \sigma= \tilde C_S^{1/S}\sigma_{new}\geq \sigma_{new}\geq \underline\sigma,
\\ & \sigma^{J+2}\delta^{1/2}= \tilde C_S^{(J+2)/S} \sigma_{new}^{J+2}\delta^{1/2}
\leq \sigma_{new}^{J+3}\delta^{1/2}\leq 1,
\end{split}
\end{equation*}
provided that $\tilde C_S^{(J+2)/S} \leq \sigma_{new}$, which is
secured by taking a sufficiently large $\underline\sigma\geq \tilde C_S^{(J+2)/S}$.
The proof is done.
\endproof

\section{Energy scaling bound for non-Euclidean films and a proof
  of Theorem \ref{th_elasticity}} \label{sec_nonEu}   

Given a solution $u\in\mathcal{C}^{1,\alpha}(\bar\omega,
\R^{d+k})$ of (\ref{II}) corresponding to the $d$-dimensional Riemannian metric
$g(\cdot, 0)_{d\times d}$ on $\bar\omega$ -- that is, the restriction
induced by the ambient $(d+k)$-dimensional metric $g$, to the midplate $\bar\omega$
-- whose existence is guaranteed by Theorem \ref{th_final},
we will explicitly define the family of immersions:
$$\{u^h\in H^1(\Omega^h,\R^{d+k})\}_{h\to 0}$$
via a suitable version of the Kirchhoff-Love extension, adapted to the metric
$g$. We will then show that, with a constant $C$ depending
only on $\omega, g, d, k,\alpha, u$, but not on $h$, there holds:
$$\mathcal{E}^h_g(u^h) \leq Ch^{\frac{4\alpha}{\alpha+1}} \quad\mbox{ for all } \; h\ll 1.$$
We use the following notation. For a matrix $g\in\R^{(d+k)\times (d+k)}_{\sym, >}$,
we denote by $g_{d\times d}$, $g_{d\times k}$,
$g_{k\times d}$, $g_{k\times k}$, the four block submatrices
determined by the first $d$ and last $k$ rows
and columns, respectively. Clearly $g_{d\times k} = (g_{k\times
  d})^T$, and both matrices $g_{d\times d}$ and $g_{k\times k}$ are
symmetric and positive definite, satisfying the formula:
\begin{equation}\label{foro}
\big((g^{-1})_{k\times k} \big)^{-1} = g_{k\times k} - g_{k\times d} (g_{d\times d})^{-1} g_{d\times k}.
\end{equation}

\bigskip

\noindent {\bf Proof of Theorem \ref{th_elasticity}}

{\bf 1.} Fix $0<\alpha<\min\big\{\frac{r+\beta}{2}, q(d,k)\big\}$,
where $q=q(d,k)\leq 1$ is the maximal regularity exponent achievable
through any variant of Theorem \ref{th_final}. In the present paper, the range
(\ref{range}) implies $q = \frac{1}{1+2d_*/k}$. Then, there exists
$u\in\mathcal{C}^{1,\alpha}(\bar\omega, \R^{d+k})$ satisfying: 
$$(\nabla u)^T\nabla u = g(\cdot,0)_{d\times d}\quad \mbox{ in } \bar\omega.$$
Without loss of
generality, $r+\beta\leq 2$.  We will use the regularization as in Lemma \ref{lem_stima}:
$$u_\epsilon=u \ast \phi_\epsilon\in\mathcal{C}^2(\bar\omega,
\R^{d+k}) \quad \mbox{ where } \; \epsilon = h^t,$$ 
for some exponent $t>0$ that will be chosen later.  It follows by a direct inspection that:
\begin{equation}\label{teraz}
\begin{split}
& \|\nabla u_\epsilon\|_{0,\alpha}\leq C,
\qquad  \|\nabla^2u_\epsilon\|_0\leq C\epsilon^{\alpha-1},
\\ & \| g(\cdot, 0)_{d\times d}\ast \phi_\epsilon - g(\cdot,
0)_{d\times d}\|_0\leq C\epsilon^{r+\beta}\leq C \epsilon^{2\alpha},
\\ & \|(\nabla u_\epsilon)^T\nabla u_\epsilon - g(\cdot, 0)_{d\times d}\ast \phi_\epsilon\|_0\leq C\epsilon^{2\alpha}, 
\end{split}  
\end{equation}
where in the last bound we used the following extension of \ref{stima4} from \cite[Lemma 2.1]{CDS}:
$$\|\nabla^{(m)}\big((fg)\ast\phi_l - (f\ast\phi_l)(g\ast\phi_l)\big)\|_0\leq C
l^{2\alpha-m}\|f\|_{0,\alpha} \|g\|_{0,\alpha},$$
with constants $C$ independent of $\epsilon$. Note that in
(\ref{teraz}) we also used that $2\alpha\leq r+\beta$.
By Corollary \ref{lem_normals} we further obtain the normal frame
$\{E_\epsilon^i\in\mathcal{C}^1(\bar\omega,\R^{d+k})\}_{i=1}^k$, for
each $\epsilon\ll 1$:
$$(\nabla u_\epsilon)^TE_\epsilon^i=0,\quad |E_\epsilon^i|=1,\quad
\langle E_\epsilon^i, E_\epsilon^j\rangle=0\quad\mbox{in }\;
\bar\omega, \quad \mbox{ for all } \; i,j=1\ldots k, \; i\neq j,$$
with the following bounds, whose uniformity is due to the fact that the corresponding
immersability constant $\gamma$ of $u_\epsilon$ and the modulus of
continuity $\|\nabla u_\epsilon\|_{0,\alpha}$ of the tangent space to $u_\epsilon(\bar\omega)$
is likewise uniform for all small $\epsilon>0$:
\begin{equation}\label{bd_norm}
\|\nabla E_\epsilon^i\|_0\leq C \|\nabla u_\epsilon\|_1\leq C \epsilon^{\alpha-1}
\quad \mbox{ for all } \; i =1\ldots k.
\end{equation}
Since the proof of Corollary \ref{lem_normals}
constructs the normal frame by inductive extensions and the Gramm-Schmidt
normalizations via Lemma \ref{lem_propa}, starting from the chosen orthonormal frame
to $\nabla u_\epsilon(x)$ at a fixed chosen $x\in\omega$, it
additionally implies the orientation preservation, namely:
\begin{equation}\label{bd_norm2}
\det \left[\begin{array}{c|ccc} \nabla u_\epsilon & E_\epsilon^1&
\ldots & E_\epsilon^k\end{array}\right] > 0 \quad\mbox{ in }\; \bar\omega.
\end{equation}

\medskip

{\bf 2. } We define the Cosserat vectors $\{b_\epsilon^i\in\mathcal{C}^1(\bar\omega,\R^{d+k})\}_{i=1}^k$
through the following formulas:
$$B_\epsilon\doteq \left[\begin{array}{cccc} b_\epsilon^1&
b_\epsilon^2&\ldots & b_\epsilon^k\end{array}\right]
= \left[\begin{array}{c|ccc} \nabla u_\epsilon & E_\epsilon^1&
\ldots & E_\epsilon^k\end{array}\right] A_\epsilon\in
\mathcal{C}^1(\bar\omega, \R^{(d+k)\times k}),$$
where the matrix of coefficients $A_\epsilon\in \mathcal{C}^1(\bar\omega, \R^{(d+k)\times k})$ is defined so that:
\begin{equation}\label{zozo}
\begin{split}  
& (\nabla u^h)^T\nabla u^h (\cdot, 0)= \left[\begin{array}{c|c} (\nabla
u_\epsilon)^T\nabla u_\epsilon& g(\cdot, 0)_{d\times k}\ast \phi_\epsilon \\
\hline g(\cdot, 0)_{k\times d}\ast \phi_\epsilon & g(\cdot, 0)_{k\times k}\ast \phi_\epsilon\end{array}\right]
\\ & \mbox{ and } \quad \det\nabla u^h(\cdot,0)  > 0 \quad \mbox{ on } \; \bar\omega,
\end{split}
\end{equation}  
for the Kirchhoff-Love extensions $u^h$ of the mollified fields $u_\epsilon$  given by:
$$u^h(x,z) = u_\epsilon(x) +\sum_{i=1}^kz_ib_\epsilon^i(x) \quad \mbox{ for all  } (x,z)\in\Omega^h.$$
To utilize the defining condition (\ref{zozo}), we first compute:
$$\nabla u^h(x,z) =  \left[\begin{array}{c|c} \nabla u_\epsilon(x) +
\sum_{i=1}^kz_i\nabla b_\epsilon^i(x) & B_\epsilon(x) \end{array}\right] 
\in \R^{(d+k)\times (d+k)} \quad \mbox{ for all  } (x,z)\in\Omega^h.$$
Equating the $d\times k$ and $k\times k$
minors in (\ref{zozo}), the first condition in there yields:
\begin{equation}\label{zozo1}
\begin{split}  
(\nabla u_\epsilon)^TB_\epsilon = g(\cdot, 0)_{d\times k}\ast
\phi_\epsilon, \qquad B_\epsilon^T B_\epsilon = g(\cdot, 0)_{k\times k}\ast \phi_\epsilon
\end{split}
\end{equation}  
Expanding with the help of the coefficient matrix fields $A_\epsilon$, we obtain:
\begin{equation*}
\begin{split}  
& (\nabla u_\epsilon)^TB_\epsilon = \left[\begin{array}{c|c} (\nabla
u_\epsilon)^T\nabla u_\epsilon &  0 \\ \end{array}\right] A_\epsilon = 
\big((\nabla u_\epsilon)^T\nabla u_\epsilon\big) (A_\epsilon)_{d\times k},
\\ & B_\epsilon^T B_\epsilon = A_\epsilon^T \mbox{diag}\big[ (\nabla
u_\epsilon)^T\nabla u_\epsilon, \mbox{Id}_{k}\big] A_\epsilon
\\ & \hspace{1cm} = (A_\epsilon)_{d\times k}^T \big((\nabla
u_\epsilon)^T\nabla u_\epsilon\big) (A_\epsilon)_{d\times k} 
+ (A_\epsilon)_{k\times k}^T (A_\epsilon)_{k\times k},
\end{split}
\end{equation*}  
so that (\ref{zozo1}) becomes:
\begin{equation}\label{zozo2}
\begin{split}
& (A_\epsilon)_{d\times k}= \big((\nabla u_\epsilon)^T\nabla u_\epsilon\big)^{-1}
\big(g(\cdot, 0)_{d\times k}\ast \phi_\epsilon\big), 
\\ & (A_\epsilon)_{k\times k}^T (A_\epsilon)_{k\times k}= g(\cdot, 0)_{k\times k}\ast \phi_\epsilon
- \big(g(\cdot, 0)_{k\times d}\ast \phi_\epsilon\big) \big((\nabla u_\epsilon)^T\nabla u_\epsilon\big)^{-1}
\big(g(\cdot, 0)_{d\times k}\ast \phi_\epsilon\big).
\end{split}
\end{equation}
We see that $(A_\epsilon)_{d\times k}\in\mathcal{C}^1(\bar\omega,\R^{d\times k})$ is indeed well
defined, as $\big((\nabla u_\epsilon)^T\nabla u_\epsilon\big)^{-1}$ is
well defined in virtue of the last two formulas in
(\ref{teraz}). Similarly, $(A_\epsilon)_{k\times k}\in\mathcal{C}^1(\bar\omega,\R^{k\times k}_{\sym,>})$ may be
taken as the unique symmetric, positive definite square root of the second
right hand side in (\ref{zozo2}), which is positive definite as its
difference from the following matrix field is infinitesimal as $\epsilon\to 0$:
$$g(\cdot, 0)_{k\times k} - g(\cdot, 0)_{k\times d} \big(g(\cdot, 0)_{d\times d}\big)^{-1}
g(\cdot, 0)_{d\times k} = \big((g(\cdot, 0)^{-1})_{k\times k}\big)^{-1}$$
by (\ref{teraz}), where above we have used the formula (\ref{foro}).
Towards the second requirement in (\ref{zozo}), we apply the column operations to obtain that
in $\bar \omega$ there holds:
\begin{equation*}
\begin{split}
\det \nabla u^h(\cdot, 0) & = \det \left[\begin{array}{c|c} \nabla u_\epsilon & B_\epsilon\end{array}\right] 
= \det  \Big[\begin{array}{c|c} \nabla u_\epsilon& \left[\begin{array}{cccc} \nabla u_\epsilon & E_\epsilon^1&
\ldots & E_\epsilon^k\end{array}\right] A_\epsilon \end{array}\Big] 
\\ & = \det \Big[\begin{array}{c|c} \nabla u_\epsilon & \left[\begin{array}{ccc} E_\epsilon^1&
\ldots & E_\epsilon^k\end{array}\right] (A_\epsilon)_{k\times k} \end{array}\Big] 
\\ & = \det \left[\begin{array}{c|ccc} \nabla u_\epsilon & E_\epsilon^1&
\ldots & E_\epsilon^k\end{array}\right] \cdot\det
\mbox{diag}\big[\mbox{Id}_d, (A_\epsilon)_{k\times k}\big] >0,
\end{split}
\end{equation*}
by (\ref{bd_norm2}) and since $(A_\epsilon)_{k\times k}$ is defined as
positive definite. Finally, by (\ref{teraz}) we directly note:
\begin{equation*}
\|A_\epsilon\|_0\leq C, \qquad\|\nabla (A_\epsilon)_{d\times k}\|_0
 +\|\nabla \big( (A_\epsilon)_{k\times k}^T (A_\epsilon)_{k\times
   k}\big)\|_0\leq C\epsilon^{\alpha-1},
\end{equation*}
which then implies in virtue of our construction, the derivative bound:
\begin{equation*}
\|\nabla A_\epsilon\|_0\leq C\epsilon^{\alpha-1}.
\end{equation*}
We therefore conclude, in view of (\ref{bd_norm}), that:
\begin{equation}\label{badabada}
\begin{split}
& \|B_\epsilon\|_0\leq C\big(\|\nabla u_\epsilon\|_0 +1\big)\|A_\epsilon\|_0\leq C,
\\ & \|\nabla B_\epsilon\|_0\leq C\big(\|\nabla^2 u_\epsilon \|_0+C\epsilon^{\alpha-1}\big)\|A_\epsilon\|_0
+ C\big(\|\nabla u_\epsilon\|_0 +1\big)\|\nabla A_\epsilon\|_0\leq C\epsilon^{\alpha-1}.
\end{split}
\end{equation}

\medskip

{\bf 3. } Since $\det \big((\nabla u^h) g^{-1/2}(\cdot, 0)\big)$ is
positive and bounded away from $0$ on $\bar\omega$, 
uniformly in $\epsilon\ll 1$, we use the polar decomposition
theorem to get, for some $SO(d+k)$-valued function $Q^h$:
$$(\nabla u^h) g^{-1/2}(\cdot, 0) = Q^h\Big(g^{-1/2}\big((\nabla
u^h)^T\nabla u^h\big)g^{-1/2}(\cdot, 0) \Big)^{1/2}.$$
Observe that, in view of (\ref{zozo}) and (\ref{teraz}):
\begin{equation}\label{badu2}
\begin{split}  
& \big\|g^{-1/2}\big((\nabla u^h)^T\nabla u^h\big)g^{-1/2}(\cdot, 0)
-\mbox{Id}_d\big\|_0 \\ & = \big\|g^{-1/2}\big((\nabla
u^h)^T\nabla u^h- g\big)g^{-1/2}(\cdot, 0) \big\|_0\\ & \leq
C\big(\|(\nabla u_\epsilon)^T\nabla u_\epsilon - g(\cdot, 0)_{d\times d}\|_0 +
\|g(\cdot,0) _{d\times d} - g(\cdot,0) _{d\times d}\ast \phi_\epsilon\|_0\big) \leq C\epsilon^{2\alpha}
\end{split}
\end{equation}
so consequently $(\nabla u^h) g^{-1/2}(x, 0)$ for all
$x\in\bar\omega$, remains in a small neighborhood of $SO(d+k)$. In that
neighborhood we may then use the assumed properties of
the density $W$. We also observe that the same statement holds for $(\nabla u^h) g^{-1/2}$ on
the entire domain $\Omega^h$,  provided that in the expression $h\epsilon^{\alpha-1} = h^{1+
  t(\alpha-1)}$ the exponent is positive and smaller than $2\alpha$, namely:
\begin{equation}\label{requi}
0<1+t(\alpha-1) \leq 2\alpha.
\end{equation}
This is because (\ref{badabada}) implies:
\begin{equation}\label{badu3}
\begin{split}
& |(\nabla u^h)g^{-1/2}(x, z) - (\nabla u^h)g^{-1/2}(x, 0)| \\ & \leq C\big( h\|\nabla
B_\epsilon\|_0+ h^{(r+\beta)\wedge 1}\big)\leq C \big(h
\epsilon^{\alpha-1}+h^{(2\alpha)\wedge 1}\big)
\leq C h \epsilon^{\alpha-1} \quad\mbox{ for all } \; (x,z)\in\Omega^h,
\end{split}
\end{equation}
We are now ready to estimate, by the frame invariance of $W$ and by (\ref{badu2}), \ref{badu3}):
\begin{equation*}%\label{Eh}
\begin{split}  
& \mathcal{E}_g^h(u^h)  \leq C \int_{\Omega^1} W\big((\nabla
u^h)g^{-1/2}(x,hz)\big) ~\mathrm{d}(x,z) \\ & =
C \int_{\Omega^1} W\big((Q^h)^T(\nabla u^h)g^{-1/2}(x,0) +
\mathcal{O}(h\epsilon^{\alpha-1}) \big) ~\mathrm{d}(x,z) 
\\ & = C \int_{\Omega^1} W\big((\mbox{Id}_d
+\mathcal{O}(\epsilon^{2\alpha}))^{1/2} + \mathcal{O}(h\epsilon^{\alpha-1}) \big) 
= C \int_{\Omega^1} W\big(\mbox{Id}_d
+\mathcal{O}(\epsilon^{2\alpha}) + \mathcal{O}(h\epsilon^{\alpha-1}) \big) 
\\ & \leq C \big(\epsilon^{2\alpha} + h\epsilon^{\alpha-1}\big)^2 =
C \big(h^{2t\alpha} + h^{1+t(\alpha-1)}\big)^2. 
\end{split}
\end{equation*}
Therefore, choosing $t=\frac{1}{\alpha+1}$ so that $2t\alpha=
1+t(\alpha-1)$ which in particular validates the requirement
(\ref{requi}), the above formula yields:
$$\inf \mathcal{E}_g^h\leq \mathcal{E}_g^h(u^h) \leq
Ch^{2t\alpha} = C h^{\frac{4\alpha}{\alpha+1}},  $$
completing the proof of Theorem \ref{th_elasticity}.
\endproof

\section{Appendix A: A proof of Lemma \ref{lem_step2}} \label{sec_appA}

In this section, we prove the key identity (\ref{sisi}) stated in
Lemma \ref{lem_step2}. Observe that the oscillatory perturbation is
introduced in a single chose normal direction $E_u$ and is
accompanied by a matching tangential correction, following the construction
inspired by \cite{Kuiper}. A related formula 
was derived for the Monge-Amp\`ere system
in \cite{lew_conv}, and should be contrasted with the 
constructions in \cite[Lemma 3.1]{CHI2} and \cite{lewpak_MA} which
contain additional terms of the form $\frac{\Gamma}{\lambda }a \nabla
a\otimes\eta$. Such terms would prevent the implementation of the 
stage induction construction
leading to estimates in Theorem \ref{thm_STA}.
Naturally, the error terms in
(\ref{sisi}) are now considerably more intricate, additionally featuring the five
orders of residual terms $\mathcal{R}_0\ldots \mathcal{R}_4$.
 
\bigskip

\noindent {\bf Proof of Lemma \ref{lem_step2}}

\smallskip

We start by calculating the gradient of the modified immersion:
\begin{equation*}
\begin{split}
\nabla \tilde u = & \;\nabla u + \Gamma'(\lambda t) a E_u\otimes \eta
+ \bar\Gamma'(\lambda t)a^2 T_u \eta\otimes\eta 
\\ & + \frac{\Gamma(\lambda t)}{\lambda}\nabla(aE_u) 
+ \frac{\bar\Gamma(\lambda t)}{\lambda}\nabla(a^2 T_u\eta)
+ \frac{\dbar\Gamma'(\lambda t)}{\lambda}aT_u\nabla a\otimes \eta
+ \frac{\dbar\Gamma(\lambda t)}{\lambda^2}\nabla \big(aT_u\nabla a).
\end{split}
\end{equation*}
Taking into account that $(\nabla u)^TT_u = \Id_d$ we obtain the
following formula, where we suppress the argument 
$\lambda t$ in $\Gamma$ and $\bar\Gamma$, and order the terms
according to powers of $\lambda$ in the denominators:
\begin{equation}\label{fair}
\begin{split}
& (\nabla \tilde u)^T\nabla \tilde u - (\nabla u)^T\nabla u \\ & = 
\Big\{2 \bar\Gamma' a^2\eta\otimes\eta + (\Gamma')^2a^2\eta\otimes\eta
+ (\bar\Gamma')^2a^4|T_u\eta|^2\eta\otimes\eta\Big\}
\\ & \quad + \Big\{\frac{2\Gamma}{\lambda} \sym\big((\nabla u)^T\nabla (aE_u)\big)
+ \frac{2\bar\Gamma}{\lambda} \sym\big((\nabla u)^T\nabla (a^2 T_u \eta)\big)
+ \frac{2\dbar\Gamma'}{\lambda} a \;\sym\big(\nabla a \otimes \eta\big)
\\ & \qquad \; \; + \frac{2\Gamma\Gamma'}{\lambda} a\;\sym\big(\nabla
(aE_u)^T(E_u\otimes\eta)\big) 
+ \frac{2\bar\Gamma\Gamma'}{\lambda}
a\;\sym\big(\nabla (a^2 T_u \eta)^T (E_u\otimes \eta)\big)  
\\ & \qquad \; \;  +\frac{2\Gamma\bar\Gamma'}{\lambda}
a^2\;\sym\big(\nabla(aE_u)^T T_u(\eta\otimes\eta) \big)
+ \frac{2\bar\Gamma\bar\Gamma'}{\lambda}
a^2\;\sym\big(\nabla( a^2T_u\eta)^TT_u(\eta\otimes\eta)\big) 
\\ & \qquad \; \;  + \frac{2\bar\Gamma'\dbar\Gamma'}{\lambda}
a^3\;\big\langle T_u\eta, T_u\nabla a\big\rangle \eta\otimes\eta\Big\}
\quad \mbox{\Large \bf +}
\end{split}
\end{equation}

\begin{equation*}
\begin{split}
& \mbox{\Large \bf +}\quad 
\Big\{\frac{2\dbar\Gamma}{\lambda^2}\sym\big((\nabla u)^T\nabla(aT_u\nabla a)\big)
+\frac{2\dbar\Gamma\Gamma'}{\lambda^2}a\;\sym\big(\nabla(aT_u\nabla
a)^T(E_u\otimes \eta) \big) 
\\ & \qquad \;\; + \frac{2\bar\Gamma'\dbar\Gamma}{\lambda^2}a^2\;
\sym\big(\nabla(aT_u\nabla a)^TT_u(\eta\otimes\eta) \big)
+\frac{\Gamma^2}{\lambda^2} \nabla(a E_u)^T\nabla (aE_u)
\\ & \qquad \;\; +
\frac{2\Gamma\bar\Gamma}{\lambda^2}\sym\big(\nabla(aE_u)^T
\nabla(a^2T_u\eta) \big) + \frac{2\Gamma\dbar\Gamma'}{\lambda^2}a\;\sym\big(\nabla(aE_u)^T
T_u(\nabla a\otimes \eta) \big)
\\ & \qquad \;\; + \frac{\bar\Gamma^2}{\lambda^2}\nabla(a^2T_u\eta)^T \nabla(a^2T_u\eta)
+\frac{2\bar\Gamma\dbar\Gamma'}{\lambda^2}a\;\sym\big(\nabla(a^2T_u\eta)^T
T_u(\nabla a\otimes \eta) \big) 
\\ & \qquad \;\; + \frac{(\dbar\Gamma')^2}{\lambda^2}a^2\; |T_u\nabla a|^2 \eta\otimes\eta \Big\}
\\ & \quad +\Big\{\frac{2\bar\Gamma\dbar\Gamma}{\lambda^3}\sym\big(\nabla(a^2T_u\eta)^T
\nabla(aT_u\nabla a) \big)
+ \frac{2\Gamma\dbar\Gamma}{\lambda^3}\sym\big(\nabla(a E_u)^T
\nabla(aT_u\nabla a) \big)
\\ & \qquad \;\; +\frac{2\dbar\Gamma\dbar\Gamma'}{\lambda^3}a\;\sym\big(\nabla(aT_u\nabla
a)^T T_u(\nabla a\otimes \eta) \big)\Big\} 
\\ & \quad + \frac{\dbar\Gamma^2}{\lambda^4} \nabla(aT_u\nabla a)^T \nabla(aT_u\nabla a). 
\end{split}
\end{equation*}
We now rewrite the first, second and fourth terms in the second parentheses, in the form:
\begin{equation*}
\begin{split}
& \frac{2\Gamma}{\lambda} \sym\big((\nabla u)^T\nabla (aE_u)\big) =
\frac{2\Gamma}{\lambda} a\; \sym\big((\nabla u)^T\nabla E_u\big)
= - \frac{2\Gamma}{\lambda} a\; \big[ \langle
\partial^2_{ij}u,E_u\rangle\big]_{i,j=1\ldots d},
\\ & \frac{2\bar\Gamma}{\lambda} \sym\big((\nabla u)^T\nabla (a^2 T_u \eta)\big)
= \frac{2\bar\Gamma}{\lambda} \Big(\sym\big(\nabla (a^2)\otimes \eta\big)
+ a^2\; \sym\big((\nabla u)^T\nabla(T_u\eta)\big)\Big)
\\ & \qquad = \frac{4\bar\Gamma}{\lambda} a\;\sym\big(\nabla a\otimes \eta\big)
- \frac{2\bar\Gamma}{\lambda} a^2\; \big[ \langle
\partial^2_{ij}u,T_u\eta\rangle\big]_{i,j=1\ldots d}, 
\\ & \frac{2\Gamma\Gamma'}{\lambda} a\;\sym\big(\nabla (aE_u)^T(E_u\otimes\eta)\big)
= \frac{2\Gamma\Gamma'}{\lambda} a\;\sym\big(\nabla a\otimes\eta\big),
\end{split}
\end{equation*}
where we have used that $(\nabla u)^TE_u=0$, $(\nabla u)^T T_u =
\Id_d$ and $(\nabla E_u)^TE_u=0$.
Similarly, we rewrite the first and the fourth term in the third
parentheses of (\ref{fair}), as follows:
\begin{equation*}
\begin{split}
& \frac{2\dbar\Gamma}{\lambda^2}\sym\big((\nabla u)^T\nabla(aT_u\nabla a)\big)
=\frac{2\dbar\Gamma}{\lambda^2}\Big(\nabla a\otimes\nabla a +
a\nabla^2a -a\;\big[ \langle \partial^2_{ij}u,T_u\nabla a \rangle\big]_{i,j=1\ldots d}\Big),
\\ & \frac{\Gamma^2}{\lambda^2} \nabla(a E_u)^T\nabla (aE_u)= 
\frac{\Gamma^2}{\lambda^2} \Big(a^2\;(\nabla E_u)^T\nabla E_u + \nabla a\otimes\nabla a\Big).
\end{split}
\end{equation*}
In conclusion and recalling the error term $\mathcal{R}$, the formula (\ref{fair}) becomes:
\begin{equation*}
\begin{split}
& (\nabla \tilde u)^T\nabla \tilde u - (\nabla u)^T\nabla u \\ &
\qquad = \big(2 \bar\Gamma' +(\Gamma')^2\big) a^2\eta\otimes\eta
- \frac{2\Gamma}{\lambda} a\; \big[ \langle \partial^2_{ij}u,E_u\rangle\big]_{i,j=1\ldots d}
+ \frac{4\bar\Gamma +2\Gamma\Gamma'+2\dbar\Gamma'}{\lambda}
a\;\sym\big(\nabla a\otimes \eta\big)
\\ & \qquad\quad + \frac{2\dbar\Gamma + \Gamma^2}{\lambda^2}\nabla a\otimes\nabla a 
+ \frac{2\dbar\Gamma}{\lambda^2} a\nabla^2a + \mathcal{R},
\end{split}
\end{equation*}
whereas we conclude (\ref{sisi}) in view of the identities:
$$ 2\bar\Gamma' + (\Gamma')^2=1, \qquad
4\bar\Gamma + 2\Gamma\Gamma'  + 2\dbar\Gamma' =0, \qquad
2\dbar\Gamma + \Gamma^2 = \tbar\Gamma.$$
The proof is done. \endproof

\section{Appendix B: proofs of propagation and existence of normal frame Lemmas
  \ref{lem_det}, \ref{lem_Tu}, \ref{lem_propa} and Corollary
  \ref{lem_normals}}\label{sec_appB} 

In this section, we collect the proofs of the lemmas on the propagation of normal
vectors stated in the preparatory section \ref{sec_step}. The
arguments follow the outline of \cite[Section 3]{lew_full2d2k} now lifted
to the general setting of arbitrary $d$ and $k\geq 2$.

\bigskip

\noindent {\bf Proof of Lemma \ref{lem_det}}

The first assertion holds as $|\partial_iu(x)|^2=\langle \nabla
u(x)^T(\nabla u(x)) e_i, e_i\rangle \in [1/\gamma, \gamma]$ for $i=1\ldots d$. For the
second assertion, note that all eigenvalues 
of the symmetric matrix $(\nabla u)^T\nabla u$ lie within the
interval $[1/\gamma, \gamma]$ for all $x\in\bar\omega$, by (\ref{immers_gamma}). 
This implies the stated determinant bound.
\endproof

\bigskip

\noindent {\bf Proof of Lemma \ref{lem_Tu}}

Writing $T_u=\frac{1}{\det((\nabla u)^T\nabla u)}(\nabla
u)\,\mbox{cof}\big((\nabla u)^T\nabla u\big)$, we easily conclude that
Lemma \ref{lem_det} implies the first assertion in (i), with $C$ depending on
$\gamma, d$. For the second assertion, we use Fa\'a di Bruno's formula
to the composition $T_u=f\circ (\nabla u)$, where the map
$f(\xi) = \frac{1}{\det(\xi^T\xi)}\xi \mbox{cof}\big(\xi^T\xi\big)$
is smooth and has all its derivatives bounded on the
domain containing the values of $\nabla u$, because of Lemma \ref{lem_det}.

\medskip

\noindent To validate the bound in (ii), we use (i) and the assumption to get:
$$\|\nabla^{(m)}T_u\|_0\leq C
\hspace{-4mm} \sum_{p_1+2p_2+\ldots mp_m=m} \; \prod_{i=1}^m\big(\bar
C \mu^i A\big)^{p_i} \leq C {\bar C}^m\mu^m A,$$
as at least one of the exponents $p_i$ in each product under the sum
above must be positive. This completes the proof.
\endproof

\bigskip

\noindent {\bf Proof of Lemma \ref{lem_propa}}

{\bf 1.} Observe that, by the first bound in Lemma \ref{lem_det}, we get:
$$ \|(\nabla v)^T\nabla v - (\nabla u)^T\nabla u\|_0\leq \|\nabla
v-\nabla u\|_0(\|\nabla u\|_0+\|\nabla v\|_0)\leq \rho_\gamma (2(d\gamma)^{1/2}+1),$$
which implies for $\rho_\gamma$ small (in function of $d,\gamma$), that:
\begin{equation}\label{a}
\frac{1}{2\gamma}\Id_{d}\leq (\nabla v)^T\nabla v\leq 2\gamma \Id_d
\quad\mbox{in } \bar\omega.
\end{equation}
In particular, $T_v$ is well defined, and since $(\nabla v)^TT_v=\Id_d$, we obtain:
\begin{equation*}
(\nabla v)^T\nu_v^i = \big((\nabla v)^T - (\nabla v-\nabla
u)^T\big)E^i_u=(\nabla u)^TE^i_u=0 \quad\mbox{ for } i=1\ldots k.
\end{equation*}
Further, applying Lemma \ref{lem_Tu} to bound $\|T_v\|_0$ by a
constant that only depends on $\gamma,d$, we get:
\begin{equation*}
\begin{split}
& \||\nu_v^i|- 1\|_0\leq \|\nu_v^i- E^i_u\|_0\leq  \|T_v\|_0\|\nabla
v-\nabla u\|_0\leq C\rho_\gamma 
\\ & \|\langle \nu_v^i, \nu_v^j\rangle-\delta_{ij}\|_0\leq
2\|T_v\|_0\|\nabla v-\nabla u\|_0 + \|T_v\|_0^2\|\nabla v-\nabla
u\|_0^2\leq C\rho_\gamma \quad \mbox{ for } i,j=1\ldots k.
\end{split}
\end{equation*}
If $\rho_\gamma$ is sufficiently small, it is clear that the fields
$\{\nu_v^i\}_{i=1}^k$ can be Gramm-Schmidt orthonormalized to the
normal frame $\{E_v^i\}_{i=1}^k$ given by the formulas in (i).

\smallskip

{\bf 2.} To prove assertion (ii), we first estimate, for all $m=0\ldots n$, $i=1\ldots k$:
\begin{equation}\label{b}
\begin{split}
& \|\nabla^{(m)}(\nu^i_v-E^i_u)\|_0 \leq
C\hspace{-2mm}\sum_{p+q+t=m}\hspace{-2mm}
\|\nabla^{(p)}(\nabla u-\nabla v)\|_0 \|\nabla^{(q)}T_v\|_0 \|\nabla^{(t)}E^i_u\|_0.
\end{split}
\end{equation}
The proof will proceed by induction on $i=1\ldots k$.
At the induction base $i=1$, we write:
\begin{equation*}
\begin{split}
& E^1_v-E^1_u = f(\nu^1_v) - f(E^1_u) = \Big(\int_0^1\nabla
f(t\nu^1_v+(1-t)E^1_u)~\mbox{d}t\Big) (\nu^1_v-E^1_u), 
\\ & \mbox{where } f(z) = \frac{z}{|z|}. 
\end{split}
\end{equation*}
For each $t\in (0,1)$, we bound $\|\nabla^{(m)}_x\big((\nabla f)\circ (E^1_u
+t(\nu_v^1-E^1_u))\big)\|_0$, using Fa\'a di Bruno's formula: 
$$\|\nabla^{(m)}_x \big((\nabla f)\circ (E^1_u
+t(\nu_v^1-E^1_u))\big)\|_0\leq C \hspace{-4mm}\sum_{p_1+2p_2+\ldots mp_m=m} \;
\prod_{s=1}^m\|\nabla^{(s)}(E_u^1+ t(\nu_v^1-E^1_u))\|_0^{p_s},$$
valid for all $m=0\ldots n$. Consequently, we get, in virtue of (\ref{b}):
\begin{equation*}
\begin{split}
&\|\nabla^{(m)}(E^1_v-E^1_u)\|_0\leq C \hspace{-6mm}\sum_{p+\sum_{s=1}^m sp_s=m}
\hspace{-4mm}\|\nabla^{(p)}(\nu_v^1-E^1_u)\|_0 \prod_{s=1}^m\Big(\|\nabla^{(s)}E_u^1\|_0^{p_s}
+ \|\nabla^{(s)}(\nu_v^1-E^1_u)\|_0^{p_s}\Big)
\\ & \leq C \hspace{-4mm}\sum_{p+\sum_{s=1}^ms(p_s+\dbar p_s)=m}\hspace{-6mm}
\|\nabla^{(p)}(\nu_v^1-E^1_u)\|_0 \prod_{s=1}^m \|\nabla^{(s)}(\nu_v^1-E^1_u)\|_0^{p_s}
\|\nabla^{(s)}E_u^1\|_0^{\dbar p_s}
\\ & \leq C \hspace{-4mm}
\sum_{p+\sum_{s=1}^m s(p_s+\bar p_s+\dbar p_s)=m\hspace{-13mm}}\hspace{-5mm}
\|\nabla^{(p)}(\nabla v -\nabla u)\|_0 \prod_{s=1}^m
\|\nabla^{(s)}(\nabla v - \nabla u)\|_0^{p_s}
\|\nabla^{(s)}T_v\|_0^{\bar p_s} \|\nabla^{(s)}E_u^1\|_0^{\dbar p_s},
\end{split}
\end{equation*}
which implies the claimed bound by Lemma \ref{lem_Tu} (i) and (\ref{a}).

\smallskip

{\bf 3.} Assume that (ii) holds up to some $i=1\ldots k-1$.
In the induction step, we aim to establish the estimate at $i+1$. We denote
$\tilde\nu^{i+1}_v= \nu^{i+1}_v - \sum_{j<i+1}\langle \nu^{i+1}_v, E^{j}_v\rangle E^j_v$,
and write:
\begin{equation}\label{goa}
\begin{split}    
\tilde\nu^{i+1}_v-E^{i+1}_u & = (\nu^{i+1}_v-E^{i+1}_u)
-\sum_{j<i+1}\langle \nu^{i+1}_v, E^{j}_v\rangle E^j_v
\\ & = (\nu^{i+1}_v-E^{i+1}_u) -\sum_{j<i+1}\langle E^{i+1}_u, E^{j}_v-E^j_u\rangle E^j_v
\end{split}
\end{equation}
as for every $j\le i$ there holds: $\langle \nu^{i+1}_v, E^{j}_v\rangle = \langle E^{i+1}_u,
E^{j}_v\rangle = \langle E^{i+1}_u, E^{j}_v- E^j_u\rangle.$ The first
term in (\ref{goa}) is estimated in (\ref{b}). For the second term, we
fix $j\leq i$ and use the induction assumption:
\begin{equation*}
\begin{split}
&\|\nabla^{(m)}(\langle E^{i+1}_u,E^j_v-E^j_u\rangle E^j_v)\|_0
\\ & \leq C\hspace{-4mm}\sum_{p+q+t=m}\hspace{-3mm} \|\nabla^{(p)}(E^j_v-E^j_u)\|_0
\big(\|\nabla^{(q)}E^j_u\|+\|\nabla^{(q)}(E^j_v-E^j_u)\|\big)\|\nabla^{(t)}E^{i+1}_u\|_0
\\ & \leq C \hspace{-4mm}
\sum_{p+t+\sum_{s=1}^m s(p_s+\bar p_s+\sum_{j\leq i}\dbar p_{j,s})=m\hspace{-30mm}}\hspace{-5mm}
\|\nabla^{(p)}(\nabla v -\nabla u)\|_0
\|\nabla^{(t)}E_u^{i+1}\|_0\prod_{s=1}^m \|\nabla^{(s)}(\nabla v-\nabla u)\|_0^{p_s}
\|\nabla^{(s)}\nabla v\|_0^{\bar p_s} \prod_{j\leq i}\|\nabla^{(s)}E_u^j\|_0^{\dbar p_{j,s}}
\end{split}
\end{equation*}
for all $m=0\ldots n$. Consequently, in view of (\ref{b}), Lemma \ref{lem_Tu} (i) and
(\ref{a}), we conclude that $\|\nabla^{(m)}(\tilde\nu_{v}^{i+1} - E^{i+1}_u)\|_0$ enjoys
exactly the same bound as above. In particular:
$$\|\tilde\nu_{v}^{i+1} - E^{i+1}_u\|_0\leq C\rho_\gamma\leq \frac{1}{2}$$ 
provided that $\rho_\gamma$ is sufficiently small. We now write, as we did in the case $i=1$:
\begin{equation*}
E^{i+1}_v-E^{i+1}_u = f(\tilde\nu^{i+1}_v) - f(E^{i+1}_u) = \Big(\int_0^1\nabla
f(t\tilde\nu^{i+1}_v+(1-t)E^{i+1}_u)~\mbox{d}t\Big) (\tilde\nu^{i+1}_v-E^{i+1}_u), 
\end{equation*}
and apply the previous estimate, to obtain:
\begin{equation*}
\begin{split}
&\|\nabla^{(m)}(E^{i+1}_v-E^{i+1}_u)\| \\ & \leq C 
\sum_{p+\sum_{s=1}^m sp_s=m\hspace{-10mm}}
\|\nabla^{(p)}(\tilde\nu_v^{i+1}-E^{i+1}_u)\|_0 \prod_{s=1}^m
\big(\|\nabla^{(s)}E_u^{i+1}\|_0^{p_s} + \|\nabla^{(s)}(\tilde\nu_v^{i+1}-E^{i+1}_u)\|_0^{p_s} \big)
\\ & \leq C \hspace{-3mm}
\sum_{p+\sum_{s=1}^m s(p_s+\dbar p_{i+1,s})=m\hspace{-13mm}}\hspace{-2mm}
\|\nabla^{(p)}(\tilde\nu_v^{i+1}-E^{i+1}_u)\|_0 \prod_{s=1}^m
\|\nabla^{(s)}(\tilde\nu_v^{i+1}-E^{i+1}_u)\|_0^{p_s} \|\nabla^{(s)}E_u^{i+1}\|_0^{\dbar p_{i+1,s}} 
\\ & \leq C \hspace{-7mm}\sum_{p+\sum_{s=1}^ms(p_s+\bar
  p_s + \sum_{j\leq i+1}\dbar p_{j,s})=m\hspace{-15mm}} \hspace{-7mm} \|\nabla^{(p)}(\nabla v-\nabla u)\|_0
\prod_{s=1}^m \|\nabla^{(s)}(\nabla v - \nabla u)\|_0^{p_s} \|\nabla^{(s)}\nabla v\|_0^{\bar p_s}
\prod_{j\leq i+1}\|\nabla^{(s)}E^j_u\|_0^{\dbar p_{j,s}},
\end{split}
\end{equation*}
which completes the proof of (ii).

\smallskip

{\bf 4.} It remains to prove (iii). There, the assumptions and (ii) yield:
\begin{equation*}
\begin{split}
&\|\nabla^{(m)}(E^i_v- E^i_u)\|\leq C \hspace{-4mm}
\sum_{p+\sum_{s=1}^m s(p_s+\bar p_s+\sum_{j\leq i} \dbar
  p_{j,s})=m\hspace{-5mm}}\hspace{-5mm}
\bar C\mu^p A \; \prod_{s=1}^m (\bar C\mu^s)^{p_s+\bar p_s  +\sum_{j\leq i}\dbar p_{j,s}}
\leq C \bar C^{m+1}\mu^m A,
\end{split}
\end{equation*}
for all $i=1\ldots k$ and $m=0\ldots n$. This ends the proof of Lemma \ref{lem_propa}.
\endproof

\bigskip

\noindent {\bf Proof of Corollary \ref{lem_normals}}

{\bf 1.} Applying the interpolation
inequality, the bounds in Lemma \ref{lem_propa} (ii) become:
\begin{equation}\label{glo3}
\begin{split}    
& \|\nabla^{(m)}(E^i_v-E^i_u)\|_0\\ & \leq C
\hspace{-7mm}\sum_{p+\sum_{s=1}^m s(p_s+\bar
  p_s + \sum_{j=1}^m\dbar p_{j,s})=m\hspace{-18mm}} \hspace{-8mm}
\|\nabla v-\nabla u\|_0^{1-p/m}\|\nabla v-\nabla u\|_m^{p/m}
\prod_{s=1}^m \Big(\|\nabla v - \nabla u\|_0^{1-s/m}\|\nabla v -
\nabla u\|_m^{s/m}\Big)^{p_s} \times
\\ & \hspace{7cm} \times \Big(\|\nabla v\|_0^{1-s/m}\|\nabla
v\|_m^{s/m}\Big)^{\bar p_s}\prod_{j=1}^m\|E_u\|_m^{s\dbar p_{j,s}/m} 
\\ & \leq C \hspace{-3mm}\sum_{p+q+t=m} \hspace{-3mm}
\|\nabla v-\nabla u\|_m^{p/m}\|\nabla v\|_m^{q/m} \|E_u\|_m^{t/m} 
\quad \mbox{ for all } \; m=1\ldots n, \; i=1\ldots k,
\end{split}
\end{equation}
in view of Lemma \ref{lem_det} and since $\rho_\gamma\leq 1$.

\smallskip  

{\bf 2.} It suffices to prove the result on $\omega=B_1$. Consider the family of immersions
$\{u_j\in\mathcal{C}^{n+1}(\bar B_1, \R^{d+k})\}_{j=1}^N$, where $u_0$ is
linear and the remaining $u_j$-s are the rescaled versions of $u$:
$$u_0(x) = \nabla u(0)x,\quad u_j(x) =
\frac{N}{j}u\big(\frac{jx}{N}\big) \quad\mbox{ for } j=1\ldots N,
\quad \mbox{ for all } \; x\in \bar B_1.$$
Since $\nabla u_j(x) =\nabla u(jx/N)$, we may choose $N\geq 1$ large enough to have:
$$\|\nabla u_{j+1}-\nabla u_j\|_0\leq \sup\big\{|\nabla u(x) -\nabla
u(y)|; \; |x-y|\leq \frac{1}{N}\big\}\leq \rho_\gamma\quad\mbox{ for
  all }\; j=0\ldots N-1,$$
where $N$ depends only on the modulus of continuity of $\nabla u$ and $\gamma$.
The normal frames $\{E^i_{u_j}\in\mathcal{C}^n(\bar B_1,R^{d+k})\}_{i=1}^k$
are now defined inductively via Lemma \ref{lem_propa}, starting from
the constant frame of  the linear immersion $u_0$. Since
all $u_j$ satisfy (\ref{immers_gamma}), the bound (\ref{glo3}) yields
for all $i=1\ldots k$:
\begin{equation*}
\begin{split}    
& \|\nabla^{(m)}(E^i_{u_{j+1}}-E^i_{u_j})\|_0\\ & \leq C
\hspace{-2mm}\sum_{p+q+t = m } \hspace{-2mm}
\|\nabla u_{j+1}-\nabla u_j\|_m^{p/m}\|\nabla u_{j+1}\|_m^{q/m} \|E_{u_j}\|_m^{t/m} 
\\ & \leq C \hspace{-2mm}\sum_{p+q+ t = m} \hspace{-2mm}
\|\nabla u\|_m^{p/m}\|\nabla u\|_m^{q/m} \|E_{u_j}\|_m^{t/m} 
\leq C \|\nabla u\|_m \;\mbox{ for } j=0\ldots N-1, \; m=1\ldots n,
\end{split}
\end{equation*}
where $C$ depends on $\gamma, d, k, m$ and $N$. Finally, since $u_N=u$, we get:
$$\|\nabla^{(m)}E^i_u\|_0 \leq\sum_{j=0}^{N-1}\|\nabla^{(m)}(E^i_{u_{j+1}}-E^i_{u_j})\|_0
\leq CN \|\nabla u\|_m \quad \mbox{ for all }\; m=1\ldots n, \; i=1\ldots n,$$
which ends the proof.
\endproof

\noindent We also observe that if $u$ above is smooth, then the normal
frame is smooth as well.


\begin{thebibliography}{99}

\bibitem{Borisov1958} {\sc Borisov, Y.}, {\em The parallel translation on a
    smooth surface. I-IV.}, Vestnik Leningrad. Univ., {\bf
    13}(4):160-171, and {\bf 13}(19):45--54, and {\bf 14}(1):34--50,
and {\bf 14}(13):20--26, (1958, 1959).

\bibitem{Borisov1965} {\sc Borisov, Y.}, {\em $C^{1,\alpha}$-isometric immersions of Riemannian
spaces}, Doklady Akad. Nauk SSSR {\bf 163}:869--871, (1965).

\bibitem{Borisov2004} {\sc Borisov, Y.}, {\em Irregular
    $C^{1,\beta}$-surfaces with analytic metric}, Sib. Mat. Zh. {\bf
    45}(1):25-–61, (2004).

\bibitem{B1} {\sc Borrelli, V., Jabrane, S., Lazarus, F. and Thibert,
    B}, {\em Flat tori in three-dimensional space and convex
    integration}, PNAS {\bf 109}: 7218–7223, (2012). 

\bibitem{B2} {\sc Bartzos, E., Borrelli, V., Denis, R., Lazarus, F.,
    Rohmer, D. and  Thibert, B}, {\em An explicit isometric reduction
    of the unit sphere into an arbitrarily small ball},
  Found. Comput. Math. {\bf 18}:1015–1042, (2018).  

\bibitem{B3} {\sc Borrelli, V., Denis, R., Lazarus, F., Theilli\`ere,
    M. and Thibert, B.}, {\em The hyperbolic plane in $\mathbb E^3$},
  Int. Math. Res. Not.  {\bf 21}:257, (2026).    

\bibitem{CHI} {\sc Cao, W., Hirsch, J. and Inauen, D.}, 
{\em $C^{1,1/3-}$ very weak solutions to the two dimensional
  Monge-Amp\`ere equation},  Calc. Var., {\bf 64}:160, (2025).

\bibitem{CHI2} {\sc Cao, W., Hirsch, J. and Inauen, D.}, 
{\em A Nash-Kuiper theorem for isometric immersions beyond Borisov's
  exponent}, {\tt arXiv:2503.13867}.

\bibitem{CHIL} {\sc Cao, W., Hirsch, J., Inauen, D. and Lewicka, M.}, 
{\em Full flexibility of the Monge-Amp\`ere system in codimension $d_*-d+1$}, 
{\tt arXiv:2603.28909} accepted in Journal Diff Eq (2026).

\bibitem{CaoIn2024} {\sc Cao W. and Inauen D.}, 
{\em Rigidity and flexibility of isometric extension.}
Comment. Math. Helv., {\bf 99}(1):39--80, (2024).

\bibitem{CS} {\sc Cao, W. and Sz\'ekelyhidi Jr., L.}, {\em Very weak
    solutions to the two-dimensional Monge-Amp\`ere equation}, 
  Sci. China Math., {\bf 62}(6): 1041--1056, (2019).

\bibitem{CaoSz2025}  {\sc Cao, W.  and Sz\'{e}kelyhidi, Jr., L.},
{\em On the isometric version of Whitney's strong embedding theorem},
Adv. Math., {\bf 460}:110040, (2025).

\bibitem {CaoSze2022} {\sc Cao, W. and Sz\'{e}kelyhidi, Jr., L.},
{\em Global {N}ash-{K}uiper theorem for compact manifolds},
J. Differential Geom. {\bf 122}(1):35--68, (2022).

\bibitem{CV} {\sc Cohn-Vossen, S.},  {\em Zwei S\"atze \"uber die
    Starrheit der Eifl\"achen}, Nachrichten G\"ottingen, pp. 125--137, (1927).

\bibitem{CDS} {\sc Conti, S., De Lellis, C. and Sz\'ekelyhidi Jr.,
    L. }, {\em h-principle and rigidity for $C^{1,\alpha}$ isometric embeddings},
Proceedings of the Abel Symposium, in: Nonlinear Partial Differential
Equations, pp. 83--116, (2010).

\bibitem{Del_masterpieces} {\sc De Lellis, C.}, {\em The masterpieces of
    John Forbes Nash Jr}, in: H. Holden and R. Piene (editors), The
  Abel Prize 2013–2017, Springer Verlag (2017).

\bibitem{DI2020} {\sc De Lellis C. and Inauen D.},
 {\em $C^{1,\alpha}$ isometric embeddings of polar caps.}
 Adv. Math., {\bf 363}:106996, (2020).

\bibitem{DIS1/5} {\sc De Lellis, C., Inauen, D. and Sz\'ekelyhidi Jr.,
    L.}, {\em A Nash-Kuiper theorem for $C^{1,\frac{1}{5}-\delta}$ immersions of
surfaces in $3$ dimensions}, Revista Matematica Iberoamericana, {\bf
34}(3):1119--1152, (2018).
 
\bibitem{GromovPdr} {\sc Gromov M.},
{\em Partial differential relations}, 
Volume 9 of ``Ergebnisse der Mathematik und ihrer
Grenzgebiete'' (3) [Results in Mathematics and Related Areas
(3)]. Springer-Verlag, Berlin, (1986).

\bibitem{Guenther} {\sc G\"unther M.},
{\em Isometric embeddings of Riemannian manifolds}. 
In ``Proceedings of the International
Congress of Mathematicians'', Vol. I, II (Kyoto, 1990), pp. 1137–-1143. Math. Soc. Japan, Tokyo, (1991). 

\bibitem{HarNir} {\sc Hartmann, P. and Nirenberg, L.}, {\em On
    spherical image maps whose Jacobians 
do not change sign}, Amer. J. Math. {\bf 81}, pp. 901--920, (1959). 

\bibitem{Her} {\sc Herglotz, G.}, {\em \"Uber die Starrheit der
    Eifl\"achen}, Abh. Math. Semin. Hansische Univ.,
{\bf 15}, pp. 127--129, (1943).

\bibitem{In2026} {\sc Inauen D.},
{\em Flexibility of codimension one $C^{1,\theta}$ isometric immersions}
{\tt arXiv:2603.08382} (2026).

\bibitem{in_lew} {\sc Inauen, D. and Lewicka, M.},
{\em The Monge-Amp\`ere system in dimension two and codimension
  three}, to appear in Revista Matematica Iberoamericana, {\tt arXiv:2501.12474} (2025).
  
\bibitem{in_lew2} {\sc Inauen, D. and Lewicka, M.},
{\em The Monge-Amp\'ere system in dimension two is fully flexible in
  codimension two}, {\tt arXiv:2504.03582} (2025).

\bibitem{Isett} {\sc Isett P.}, {\em A proof of Onsager’s conjecture.}
Ann. of Math. (2), {\bf 188}(3):871--963, (2018).

\bibitem{Jaco} {\sc Jacobowitz, H.}, {\em Implicit function theorems and isometric embeddings},
Annals of Mathematics, {\bf 95}(2):191--225, (1972).

\bibitem{Kallen} {\sc K\"all\'en, A.}, {\em Isometric embedding of a
    smooth compact manifold with a metric of 
  low regularity}, Ark. Mat. {\bf 16}(1): 29--50, (1978). 
  
\bibitem{Kuiper} {\sc Kuiper, N.}, {\em On {$C^1$}-isometric imbeddings. {I}, {II}.},
Indag. Math., {\bf 17}:545--556 and 683--689,
Nederl. Akad. Wetensch. Proc. Ser. A {\bf 58}, (1955).

\bibitem{lew_conv} {\sc Lewicka, M.}, 
{\em The Monge-Amp\`ere system: convex integration in arbitrary
  dimension and codimension}, SIAM J. Math. Analysis, {\bf 57}(1): 
601--636, (2025). 

\bibitem{lew_improved} {\sc Lewicka, M.},
{\em The Monge-Amp\`ere system: convex integration with
  improved regularity in dimension two and arbitrary codimension},
Diff. Int. Equations, {\bf 38}(3/4):141--174, (2025). 

\bibitem{lew_improved2} {\sc Lewicka, M.},
{\em The Monge-Amp\`ere system in dimension two: a regularity
  improvement}, J. Funct. Anal., {\bf 289}(8):111064, (2025).

\bibitem{lew_full2d2k} {\sc Lewicka, M.},
{\em Full flexibility of isometric immersions of metrics with low H\"older
  regularity in Poznyak theorem's dimension}, {\tt arXiv:2511.16305}.

\bibitem{lew_book} {\sc Lewicka, M.},
{\em Calculus of Variations on Thin Prestressed Films: Asymptotic
  Methods in Elasticity}, Progress in Nonlinear Differential Equations
and Their Applications, {\bf 101}, Birkh\"auser, (2022).

\bibitem{lewlu} {\sc Lewicka, M. and Lucic, D.},
{\em Dimension reduction for thin films with transversally varying
  prestrian: the oscillatory and the non-oscillatory case},
Communications on Pure and Applied Mathematics, {\bf 73}(9):1880--1932, (2020).

\bibitem{LeOchPa} {\sc Lewicka, M., Ochoa, P. and Pakzad, R.},
{\em Variational models for prestrained plates with Monge-Amp\'ere
  constraint}, Diff. Integral Equations, {\bf 28}(9-10):861--898, (2015).

\bibitem{lewpak_MA} {\sc Lewicka, M.  and Pakzad, M.},
{\em Convex integration for the Monge-Amp\`ere equation in two
  dimensions}, Analysis and PDE, {\bf 10}(3):695--727, (2017).

\bibitem{Nash1} {\sc Nash, J.}, {\em The imbedding problem for Riemannian
    manifolds}, Ann. Math., {\bf 63}:20--63, (1956).

\bibitem{Nash2} {\sc Nash, J.}, {\em $\mathcal{C}^1$ isometric
    imbeddings}, Ann. Math., {\bf 60}:383--396, (1954).

\bibitem{Pogo1} {\sc Pogorelov, A.}, {\em Surfaces with bounded extrinsic curvature} (Russian),
Kharhov, (1956).

\bibitem{Pogo2} {\sc Pogorelov, A.}, {\em Extrinsic geometry of convex surfaces}, Translation of
Mathematical Monographs, {\bf 35}, American Math. Soc., (1973).
  
\bibitem{Poz} {\sc Poznjak, \`E.\~G.},
{\em Isometric imbedding of two-dimensional {R}iemannian metrics in {E}uclidean spaces},  
Uspehi Mat. Nauk, {\bf 28}(4):47--76, (1973).

\bibitem{Weyl} {\sc Weyl, W.}, {\em \"Uber die Bestimmung einer
    geschlossenen konvexen 
Fl\"ache durch ihr Linienelement}, Vierteljahrsschrift der
naturforschenden Gesellschaft, Zurich, {\bf 61}, pp. 40-72, (1916).

\end{thebibliography}
\end{document}